\documentclass[12pt]{amsart}
\usepackage{color}
\usepackage{enumerate}
\usepackage{graphics}
\usepackage[colorlinks, urlcolor=blue]{hyperref}
\usepackage{mathrsfs}
\usepackage{amssymb}
\usepackage[all]{xy}
\usepackage{amsthm}

\newtheorem{theorem}{Theorem}[section] 
\newtheorem*{theorem*}{Theorem}
\newtheorem*{problem*}{Problem}
\newtheorem{lemma}[theorem]{Lemma}
\newtheorem{proposition}[theorem]{Proposition}

\newtheorem{definition}[theorem]{Definition}
 
\theoremstyle{remark}

\newtheorem{remark}[theorem]{Remark}

   \renewcommand{\H}{\mbox{${\mathcal H}$}}

   \newcommand{\N}{\mbox{${\mathbb N}$}}

   \def\Sc{{\mathcal{S}}}

   \def\tr{{\mathrm{Tr}\,}}

\begin{document}

\title[Perturbation theory and higher order $\Sc^p$-differentiability of operator functions]
{Perturbation theory and higher order $\Sc^p$-differentiability of operator functions}

\author[C. Coine]{Cl\'ement Coine}
\email{clement.coine1@gmail.com}

\address{School of Mathematics and Statistics, Central South University, Changsha 410085,
People’s Republic of China}

\subjclass[2000]
{47B49, 47B10, 47A55, 46L52}

\keywords{Differentiation of operator functions, perturbation theory}

\date{\today}

\maketitle

\begin{abstract}
We establish, for $1 < p < \infty$, higher order $\mathcal{S}^p$-differentiability results of the function $\varphi : t\in \mathbb{R} \mapsto f(A+tK) - f(A)$ for selfadjoint operators $A$ and $K$ on a separable Hilbert space $\mathcal{H}$ with $K$ element of the Schatten class $\mathcal{S}^p(\mathcal{H})$ and $f$ $n$-times differentiable on $\mathbb{R}$. We prove that if either $A$ and $f^{(n)}$ are bounded or $f^{(i)}, 1 \leq i \leq n$ are bounded, $\varphi$ is $n$-times differentiable on $\mathbb{R}$ in the $\mathcal{S}^p$-norm with bounded $n$th derivative. If $f\in C^n(\mathbb{R})$ with bounded $f^{(n)}$, we prove that $\varphi$ is $n$-times continuously differentiable on $\mathbb{R}$. We give explicit formulas for the derivatives of $\varphi$, in terms of multiple operator integrals. As for application, we establish a formula and $\mathcal{S}^p$-estimates for operator Taylor remainders for a more extensive class of functions. These results are the $n$th order analogue of the results of \cite{KPSS}. They also extend the results of \cite{CLSS} from $\mathcal{S}^2(\mathcal{H})$ to $\mathcal{S}^p(\mathcal{H})$ and the results of \cite{LMS} from $n$-times continuously differentiable functions to $n$-times differentiable functions $f$.
\end{abstract}

\section{Introduction}

Let $\mathcal{H}$ be a separable Hilbert space and let, for any $1 < p < \infty$, $\mathcal{S}^p(\mathcal{H})$ be the Schatten class of order $p$ on $\mathcal{H}$.
Let $A$ be a (possibly unbounded) selfadjoint operator on $\mathcal{H}$ and let $K = K^* \in \mathcal{S}^p(\mathcal{H})$. Let $f : \mathbb{R} \to \mathbb{C}$ be a Lipschitz function. We let $\varphi$ to be the function defined on $\mathbb{R}$ by
$$
\varphi : t \in \mathbb{R} \mapsto f(A+tK) - f(A) \in \mathcal{S}^p(\mathcal{H}).
$$
In this paper, we prove higher order $\mathcal{S}^p$-differentiability results for $\varphi$ in the case of $n$-times differentiable functions $f$ with bounded (possibly discontinuous) $n$th derivative.

The study of differentiabily of $\varphi$ was initiated in \cite{DK1956} where it was shown that if $A$ and $K$ are bounded selfadjoint operators and $f \in C^2(\mathbb{R})$, $\varphi$ is differentiable in the operator norm with
$$
\varphi'(t) = \left[ \Gamma^{A+tK, A+tK}(f^{[1]}) \right](K), t\in \mathbb{R},
$$
where $\Gamma^{A+tK, A+tK}(f^{[1]})$ is a double operator integral associated with $f^{[1]}$, the divided difference of first order of $f$. See Section $\ref{MOI}$ for more details. This result was extended in \cite{BS3} and later in \cite{Peller1990} where it is proved that this result holds true for any $f$ in the Besov space $B^1_{\infty, 1}(\mathbb{R})$ and any selfadjoint operator $A$. Note that the conditions $f \in C^1(\mathbb{R})$ and $A$ bounded are not sufficient to ensure the differentiability of $\varphi$ in the operator norm, see \cite{Far}. However, in the case $K\in \mathcal{S}^p(\mathcal{H}), 1<p<\infty$, it is shown in \cite{KPSS} that if $f$ is differentiable on $\mathbb{R}$ with bounded derivative, then $\varphi$ is $\mathcal{S}^p$-differentiable on $\mathbb{R}$.

The question of higher order differentiability of $\varphi$ was studied in \cite{Stek}. Under certain assumptions on $f$, $\varphi$ is $n$-times differentiable for the operator norm and the derivatives of $\varphi$ are represented as multiple operator integrals. This result was extended in \cite{Peller2006} to any $f$ in the intersection $B^1_{\infty, 1}(\mathbb{R}) \cap B^n_{\infty, 1}(\mathbb{R})$ of Besov classes. In \cite{ACDS}, higher order differentiability of $\varphi$ is established in the symmetric operator ideal norm when $f$ is in the Wiener space $W_{n+1}(\mathbb{R})$. In the special case $p=2$, it is proved that if $f\in C^n(\mathbb{R})$ has bounded derivatives $f^{(i)}, 1\leq i\leq n$, $\varphi$ is $n$-times continuously $\mathcal{S}^2$-differentiable on $\mathbb{R}$, see \cite{CLSS}. For other values of $1 < p < \infty$, it is shown in \cite{LMS} that if $f\in C^n(\mathbb{R})$ has bounded derivatives, then $\varphi$ is $n$-times $\mathcal{S}^p$-differentiable. Moreover, for $1\leq p < \infty$, \cite[Theorem 4.1]{LMS} shows that for functions in $B^1_{\infty, 1}(\mathbb{R}) \cap B^n_{\infty, 1}(\mathbb{R})$, $\varphi$ is $n$-times $\mathcal{S}^p$-continuously differentiable.

Our main result is the following. Let $1<p<\infty$, $n\in \mathbb{N}$ and let $K=K^* \in \mathcal{S}^p(\mathcal{H})$. We prove that if $f$ is $n$-times differentiable on $\mathbb{R}$ with bounded (possibly discontinuous) $n$th derivative $f^{(n)}$, then for any bounded selfadjoint operator $A$, $\varphi$ is $n$-times differentiable on $\mathbb{R}$ and for any $1\leq k\leq n$, 
\begin{equation}\label{diffenrentialintro}
\dfrac{1}{k!} \varphi^{(k)}(t) = \left[ \Gamma^{A+tK, A+tK, \ldots, A+tK}(f^{[k]}) \right] (K, \ldots, K), t\in \mathbb{R}.
\end{equation}
This representation of $\varphi^{(k)}$ has been obtained for smaller classes of functions, see for instance \cite{ACDS,CLSS,Peller2006,Stek}.
In the case when $A$ is unbounded, we prove that if $f$ is $n$-times differentiable on $\mathbb{R}$ and has bounded derivatives $f^{(i)}, 1\leq i\leq n$, then so does $\varphi$. Namely, we show that $\varphi$ is $n$-times $\mathcal{S}^p$-differentiable on $\mathbb{R}$ with bounded derivatives $\varphi^{(j)}, 1 \leq j \leq n$, and that Formula \eqref{diffenrentialintro} holds. This is $n$th order analogue of \cite[Theorem 7.13]{KPSS}. It significantly improves the previous results on higher order differentiabily of operator functions in Schatten norms.

With Formula \eqref{diffenrentialintro}, we deduce a representation of Taylor remainders
$$
f(A+K)-f(A) - \sum_{k=1}^{n-1} \frac{1}{k!}\varphi^{(k)}(0)
$$
as a multiple operator integral and deduce an $\mathcal{S}^p$-estimate, which generalizes the estimate obtained in \cite{LMS}.

To obtain these results, we will establish important properties of multiple operator integrals. We choose the construction of operator integrals developed in \cite{CLS}. For any selfadjoint operators $A_1, \ldots, A_n$ and any bounded Borel function $\phi$ on $\mathbb{R}^{n}$, the multiple operator integral $\Gamma^{A_1, \ldots, A_n}(\phi)$ is a continuous $(n-1)$-linear mapping defined on the product of $n-1$ copies of $\mathcal{S}^2(\mathcal{H})$ and valued in $\mathcal{S}^2(\mathcal{H})$. We obtain a continuous operator
$
\Gamma^{A_1,A_2, \ldots, A_n} : L^{\infty}\left(\prod_{i=1}^n
\lambda_{A_i}\right) \rightarrow \mathcal{B}_{n-1}(\mathcal{S}^2(\mathcal{H}) \times \mathcal{S}^2(\mathcal{H})
\times \cdots \times \mathcal{S}^2(\mathcal{H}), \mathcal{S}^2(\mathcal{H}))
$
for some positive and finite measures $\lambda_{A_i}, 1 \leq i \leq n$. The advantage of this construction is the property of $w^*$-continuity of $\Gamma^{A_1,A_2, \ldots, A_n}$. It allows to reduce some computations to functions with separated variables, for which certain equations are straightforward to establish. In Section $\ref{HOPF}$, we extend a result on the $\mathcal{S}^p$-boundedness of multiple operator integrals associated to divided differences. Our main result will be proved by induction on $n$. To do so, we will first establish an important higher order perturbation formula allowing to express a difference of operator integrals associated to $f^{[n-1]}$ as a multiple operator integral associated to $f^{[n]}$. This formula will be fundamental to prove the existence of the $n$th derivative of $\varphi^{(n)}$ if $\varphi^{(n-1)}$ is known, as well as the representation of the derivatives of $\varphi$ as a multiple operator integral. Then, by the use of the lemmas proved in Section $\ref{SubAux}$, our proof will rest on the approximation of the operator $K$, allowing to simplify the expression of the multiple operator integrals involved.

We use the following notations. We let $(\mathcal{S}^p(\mathcal{H}))_{\text{sa}}$ (respectively $(\mathcal{B}(\mathcal{H})_{\text{sa}}$) to be the subspace of $\mathcal{S}^p(\mathcal{H})$ (respectively $\mathcal{B}(\mathcal{H})$) consisting of selfadjoint operators. We let $\text{Bor}(\mathbb{R})$ to be the space of bounded Borel functions from $\mathbb{R}$ into $\mathbb{C}$. For any $m\in \mathbb{N}$, we let $C_b(\mathbb{R}^m)$ to be space of continuous and bounded functions on $\mathbb{R}^{m}$ and $C_0(\mathbb{R}^m)$ to be the subspace of $C_b(\mathbb{R}^m)$ of continuous functions on $\mathbb{R}^m$ vanishing at infinity. For any $n\geq 1$, we let $C^n(\mathbb{R})$ to be the space of $n$-times continuously differentiable functions from $\mathbb{R}$ to $\mathbb{C}$. Finally, we let $D^n(\mathbb{R}, \mathcal{S}^p(\mathcal{H}))$ (respectively $C^n(\mathbb{R}, \mathcal{S}^p(\mathcal{H}))$) to be the space of $n$-times differentiable (respectively continuously differentiable) functions $\phi : \mathbb{R} \to \mathcal{S}^p(\mathcal{H})$ with derivatives denoted by $\phi^{(j)} : \mathbb{R} \to \mathcal{S}^p(\mathcal{H}), j=1, \ldots, n$.

\section{Multiple operator integration}\label{MOI}

In this section, we recall the definition of multiple operator integrals that we will use throughout the paper and give important properties that will be key to prove our main results.

\subsection{Multiple operator integrals associated to selfadjoint operators}

The following definition of multiple operator integration was developed in \cite{CLS}. It is based on the construction of \cite{Pav}. Several other constructions exist, see e.g. \cite{ACDS,BS1,DK1956,Peller2006,PSS-SSF}. The first advantage of this approach is that it allows us to integrate any bounded Borel function, in particular certain discontinuous ones, as it will be the case in this paper. The second advantage is the property of $w^*$-continuity, which allows to simplify many computations.

Let $n\in \mathbb{N}, n\geq 1$ and let $E_1, \ldots, E_n, E$ be Banach spaces. We denote by $\mathcal{B}_n(E_1 \times \cdots \times E_n, E)$ the space of $n$-linear continuous mappings from $E_1 \times \cdots \times E_n$ into $E$, that is, the space of $n$-linear mappings $T : E_1 \times \cdots \times E_n \to E$ such that
$$
\|T\|_{\mathcal{B}_n(E_1 \times \cdots \times E_n, E)} := \sup_{\|e_i\| \leq 1, 1\leq i \leq n} ~ \|T(e_1, \ldots, e_n)\| < \infty.
$$
In the case when $E_1  = \cdots = E_n = E$, we will simply denote $\mathcal{B}_n(E_1 \times \cdots \times E_n, E)$ by $\mathcal{B}_n(E)$.

Let $A$ be a (possibly unbounded) selfadjoint operator in $\mathcal{H}$. Denote its spectrum by $\sigma(A)$ and its measure spectral by $E^A$. Let $\lambda_A$ be a scalar-valued spectral measure for $A$, that is, a positive finite measure on the Borel subsets of $\sigma(A)$ such that $\lambda_A$ and $E^A$ have the same sets of measure zero. We refer to \cite[Section 15]{Conway} and \cite[Section 2.1]{CLS} for more details. For any bounded Borel function $f : \mathbb{R} \to \mathbb{C}$, we define $f(A) \in \mathcal{B}(\mathcal{H})$ by
$$
f(A):=\int_{\sigma(A)} f(t) \ \text{d}E^A(t),
$$
and this operator only depends on the class of $f$ in $L^{\infty}(\lambda_A)$. Moreover, according to \cite[Theorem 15.10]{Conway}, we obtain a $w^*$-continuous $*$-representation
$$
f \in L^{\infty}(\lambda_A) \mapsto f(A) \in \mathcal{B}(\mathcal{H}).
$$
Let $n\in\N, n\geq 2$ and let $A_1, A_2, \ldots, A_n$ be selfadjoint operators in $\mathcal{H}$ with scalar-valued spectral measures $\lambda_{A_1}, \ldots, \lambda_{A_n}$. We let
\begin{equation*}
\Gamma^{A_1,A_2, \ldots, A_n} : L^{\infty}(\lambda_{A_1}) \otimes \cdots \otimes L^{\infty}(\lambda_{A_n}) \rightarrow \mathcal{B}_{n-1}(\mathcal{S}^2(\mathcal{H}))
\end{equation*}
to be the unique linear map such that for any $f_i \in L^{\infty}(\lambda_{A_i}), i=1, \ldots, n$ and for any $X_1, \ldots, X_{n-1} \in \mathcal{S}^2(\mathcal{H})$,
\begin{align}\label{MOItensor}
\left[\Gamma^{A_1,A_2, \ldots, A_n}(f_1\otimes\cdots\otimes f_n)\right]
& (X_1,\ldots, X_{n-1})\\ \nonumber
& =f_1(A_1)X_1f_2(A_2) \cdots f_{n-1}(A_{n-1})X_{n-1}f_n(A_n).
\end{align}
Note that $\mathcal{B}_{n-1}(\mathcal{S}^2(\mathcal{H}))$ is a dual space, see \cite[Section 3.1]{CLS} for details. According to \cite[Theorem 5 and Proposition 6]{CLS}, $\Gamma^{A_1,A_2, \ldots, A_n}$ extends to a unique $w^*$-continuous contraction still denoted by
$$
\Gamma^{A_1,A_2, \ldots, A_n} : L^{\infty}\left(\prod_{i=1}^n
\lambda_{A_i}\right) \longrightarrow
\mathcal{B}_{n-1}(\mathcal{S}^2(\mathcal{H})).
$$

\begin{definition}
For $\phi \in L^{\infty}\left(\prod_{i=1}^n\lambda_{A_i}\right)$, the transformation $\Gamma^{A_1,A_2, \ldots, A_n}(\phi)$ is called a multiple operator integral associated to $A_1, A_2, \ldots, A_n$ and $\phi$.
\end{definition}

The $w^*$-continuity of $\Gamma^{A_1,A_2, \ldots, A_n}$ means that if a net $(\phi_i)_{i\in I}$ in $L^{\infty}\left(\prod_{i=1}^n
\lambda_{A_i}\right)$ converges to $\phi \in L^{\infty}\left(\prod_{i=1}^n \lambda_{A_i}\right)$ in the $w^*$-topology, then for any $X_1, \ldots, X_{n-1} \in \mathcal{S}^2(\mathcal{H})$, the net
$$
\bigl(\left[\Gamma^{A_1,A_2, \ldots, A_n}(\phi_i)\right](X_1,\ldots, X_{n-1})\bigr)_{i\in I}
$$
converges to $\left[\Gamma^{A_1,A_2, \ldots, A_n}(\phi)\right](X_1,\ldots, X_{n-1})$ weakly in $\mathcal{S}^2(\mathcal{H})$.\\

Let $\alpha_1, \ldots, \alpha_{n-1}, \alpha \in [1, \infty)$ and $\phi \in L^{\infty}\left(\prod_{i=1}^n\lambda_{A_i}\right)$. We will write $\Gamma^{A_1,A_2, \ldots, A_n}(\phi)$ $\in \mathcal{B}_{n-1}(\mathcal{S}^{\alpha_1} \times \cdots \times \mathcal{S}^{\alpha_{n-1}}, \mathcal{S}^{\alpha})$
if the multiple operator integral $\Gamma^{A_1,A_2, \ldots, A_n}(\phi)$ defines a bounded $(n-1)$-linear mapping
$$
\Gamma^{A_1,A_2, \ldots, A_n}(\phi) : \left( \mathcal{S}^2(\mathcal{H}) \cap \mathcal{S}^{\alpha_1}(\mathcal{H}) \right) \times \cdots \times \left( \mathcal{S}^2(\mathcal{H}) \cap \mathcal{S}^{\alpha_{n-1}}(\mathcal{H}) \right) \rightarrow \mathcal{S}^{\alpha}(\mathcal{H}),
$$
where $\mathcal{S}^2(\mathcal{H}) \cap \mathcal{S}^{\alpha_i}(\mathcal{H})$ is equipped with the $\|.\|_{\alpha_i}$-norm. By density of $\mathcal{S}^2(\mathcal{H}) \cap \mathcal{S}^{\alpha_i}(\mathcal{H})$ into $\mathcal{S}^{\alpha_i}(\mathcal{H})$, this mapping has a (necessarily) unique extension
$$
\Gamma^{A_1,A_2, \ldots, A_n}(\phi) : \mathcal{S}^{\alpha_1}(\mathcal{H}) \times \cdots \times \mathcal{S}^{\alpha_{n-1}}(\mathcal{H}) \rightarrow \mathcal{S}^{\alpha}(\mathcal{H}),
$$
which justifies the notation.

In the case when $\alpha_1 = \cdots= \alpha_{n-1} = \alpha$, we will simply write $\Gamma^{A_1,A_2, \ldots, A_n}(\phi) \in \mathcal{B}_{n-1}(\mathcal{S}^{\alpha}(\mathcal{H}))$.

\begin{remark}\label{Ineqmulti}
Let $\alpha_1, \ldots, \alpha_{n-1}, \alpha \in [1, \infty)$, let $n \geq 1$, $A_1, \ldots, A_n$ be selfadjoint operators on $\mathcal{H}$, $\phi \in L^{\infty}(\lambda_{A_1} \times \cdots \lambda_{A_n})$ and assume that $\Gamma^{A_1, \ldots, A_n}(\phi) \in \mathcal{B}_{n-1}(\mathcal{S}^{\alpha_1} \times \cdots \times \mathcal{S}^{\alpha_{n-1}}, \mathcal{S}^{\alpha})$. Let $0 < \epsilon < 1$, let $X_1, \ldots, X_{n-1}, Y_1, \ldots, Y_{n-1}$ where for any $1 \leq i \leq n-1$, $X_i, Y_i\in \mathcal{S}^{\alpha_i}(\mathcal{H})$ with $\|X_i - Y_i \|_{\alpha_i} \leq \epsilon$. By multilinearity of multiple operator integrals, it is easy to see that there exists a constant $C > 0$ depending only on $n, \|\Gamma^{A_1, \ldots, A_n}(\phi)\|_{\mathcal{B}_{n-1}(\mathcal{S}^{\alpha_1} \times \cdots \times \mathcal{S}^{\alpha_{n-1}}, \mathcal{S}^{\alpha})}, \|X_1\|_{\alpha_1},$ $\ldots, \|X_{n-1}\|_{\alpha_{n-1}}$ (or similarly, on $n, \|\Gamma^{A_1, \ldots, A_n}(\phi)\|_{\mathcal{B}_{n-1}(\mathcal{S}^{\alpha_1} \times \cdots \times \mathcal{S}^{\alpha_{n-1}}, \mathcal{S}^{\alpha})}, \|Y_1\|_{\alpha_i},$ $\ldots, \|Y_{n-1}\|_{\alpha_{n-1}}$) such that
$$
\| \left[ \Gamma^{A_1, \ldots, A_n}(\phi) \right](X_1, \ldots, X_{n-1}) - \left[ \Gamma^{A_1, \ldots, A_n}(\phi) \right](Y_1, \ldots, Y_{n-1}) \|_{\alpha} \leq C \epsilon. 
$$
\end{remark}

The following result will be used to prove the $\mathcal{S}^p$-boundedness of certain multiple operator integrals as well as to establish identities.

\begin{lemma}\label{LemmeUB}
Let $\alpha_1, \ldots, \alpha_{n-1}, \alpha \in (1, \infty)$, let $n \geq 1$, $A_1, \ldots, A_n$ be selfadjoint operators in $\mathcal{H}$ and $(\varphi_k)_{k\geq 1}, \varphi \in L^{\infty}(\lambda_{A_1} \times \cdots \lambda_{A_n})$. Assume that $(\varphi_k)_k$ is $w^*$-convergent to $\varphi$ and that $\left(\Gamma^{A_1, \ldots, A_n}(\varphi_k) \right)_{k\geq 1} \subset \mathcal{B}_{n-1}(\mathcal{S}^{\alpha_1} \times \cdots \times \mathcal{S}^{\alpha_{n-1}}, \mathcal{S}^{\alpha})$ is bounded. Then $\Gamma^{A_1, \ldots, A_n}(\varphi) \in \mathcal{B}_{n-1}(\mathcal{S}^{\alpha_1} \times \cdots \times \mathcal{S}^{\alpha_{n-1}}, \mathcal{S}^{\alpha})$ with
\begin{align*}
& \| \Gamma^{A_1, \ldots, A_n}(\varphi)\|_{\mathcal{B}_{n-1}(\mathcal{S}^{\alpha_1} \times \cdots \times \mathcal{S}^{\alpha_{n-1}}, \mathcal{S}^{\alpha})} \\
& \ \ \ \ \leq \liminf_k \| \Gamma^{A_1, \ldots, A_n}(\varphi_k)\|_{\mathcal{B}_{n-1}(\mathcal{S}^{\alpha_1} \times \cdots \times \mathcal{S}^{\alpha_{n-1}}, \mathcal{S}^{\alpha})}
\end{align*}
and for any $X_i \in \mathcal{S}^{\alpha_i}(\mathcal{H}), 1 \leq i \leq n-1$,
$$
\left[ \Gamma^{A_1, \ldots, A_n}(\varphi_k) \right](X_1, \ldots, X_{n-1}) \underset{k \to \infty}{\longrightarrow} \left[ \Gamma^{A_1, \ldots, A_n}(\varphi) \right](X_1, \ldots, X_{n-1})
$$
weakly in $\mathcal{S}^{\alpha}(\mathcal{H})$.
\end{lemma}

\begin{proof}
Let $X_i \in \mathcal{S}^2(\mathcal{H}) \cap \mathcal{S}^{\alpha_i}(\mathcal{H}), 1 \leq i \leq n-1$, and let $Y$ be a finite-rank operator on $\mathcal{H}$ such that $\|Y\|_{\alpha'} \leq 1$. Let
$$\gamma := \liminf_k \| \Gamma^{A_1, \ldots, A_n}(\varphi_k)\|_{\mathcal{B}_{n-1}(\mathcal{S}^{\alpha_1} \times \cdots \times \mathcal{S}^{\alpha_{n-1}}, \mathcal{S}^{\alpha})}.$$
By $w^*$-continuity of multiple operator integrals and the assumptions of the Lemma we have
\begin{align*}
& | \tr(\left[\Gamma^{A_1,A_2, \ldots, A_n}(\varphi)\right](X_1,\ldots,X_{n-1})Y)| \\
& = \liminf_k | \tr(\left[\Gamma^{A_1,A_2, \ldots, A_n}(\varphi_k)\right](X_1,\ldots,X_{n-1})Y)| \\
& \leq \liminf_k \| \left[\Gamma^{A_1,A_2, \ldots, A_n}(\varphi_k)\right](X_1,\ldots,X_{n-1}) \|_{\alpha} \|Y\|_{\alpha'} \\
& \leq \gamma \|X_1\|_{\alpha_1} \cdots \|X_{n-1}\|_{\alpha_{n-1}}.
\end{align*}
This inequality holds true for any finite-rank operator $Y$ on $\mathcal{H}$ with $\|Y\|_{\alpha'} \leq 1$, hence
\begin{equation}\label{Proofineqp}
\| \left[\Gamma^{A_1,A_2, \ldots, A_n}(\varphi)\right](X_1,\ldots,X_{n-1}) \|_{\alpha} \leq \gamma \|X_1\|_{\alpha_1} \cdots \|X_n\|_{\alpha_{n-1}}.
\end{equation}
This implies that $\Gamma^{A_1, \ldots, A_n}(\varphi) \in \mathcal{B}_{n-1}(\mathcal{S}^{\alpha_1} \times \cdots \times \mathcal{S}^{\alpha_{n-1}}, \mathcal{S}^{\alpha})$ with
$$
\| \Gamma^{A_1, \ldots, A_n}(\varphi)\|_{\mathcal{B}_{n-1}(\mathcal{S}^{\alpha_1} \times \cdots \times \mathcal{S}^{\alpha_{n-1}}, \mathcal{S}^{\alpha})} \leq \gamma.
$$
Let $0 < \epsilon < 1$. For any $1 \leq i \leq n-1$, let $X_i \in \mathcal{S}^{\alpha_i}(\mathcal{H}), \tilde{X}_i \in \mathcal{S}^2(\mathcal{H}) \cap \mathcal{S}^{\alpha_i}(\mathcal{H})$ such that $\| X_i - \tilde{X}_i \|_{\alpha_i} \leq \epsilon$. Let $Z \in \mathcal{S}^{\alpha'}(\mathcal{H})$ and $Y$ be a finite-rank operator on $\mathcal{H}$ such that $\| Z - Y \|_{\alpha'} \leq \epsilon$. Write, for any $k\geq 1$,
$$\Gamma_{k,X} = \left[\Gamma^{A_1,A_2, \ldots, A_n}(\varphi_k)\right](X_1,\ldots, X_{n-1})$$
and
$$\tilde{\Gamma}_{k,X} = \left[\Gamma^{A_1,A_2, \ldots, A_n}(\varphi_k)\right](\tilde{X}_1, \ldots, \tilde{X}_{n-1}).
$$
Similarly, write
$$\Gamma_{X} = \left[\Gamma^{A_1,A_2, \ldots, A_n}(\varphi)\right](X_1,\ldots,X_{n-1})$$
and
$$\tilde{\Gamma}_X = \left[\Gamma^{A_1,A_2, \ldots, A_n}(\varphi)\right](\tilde{X}_1, \ldots, \tilde{X}_{n-1}).
$$
Since $\left(\Gamma^{A_1, \ldots, A_n}(\varphi_k) \right)_{k\geq 1} \subset \mathcal{B}_{n-1}(\mathcal{S}^{\alpha_1} \times \cdots \times \mathcal{S}^{\alpha_{n-1}}, \mathcal{S}^{\alpha})$ is bounded, we can set $C' := \sup_j \|\Gamma^{A_1, \ldots, A_n}(\varphi_j)\|_{\mathcal{B}_{n-1}(\mathcal{S}^{\alpha_1} \times \cdots \times \mathcal{S}^{\alpha_{n-1}}, \mathcal{S}^{\alpha})}$. By Remark $\ref{Ineqmulti}$, there exists a constant $C > 0$ depending only on $n, C', \|X_1\|_{\alpha_1}, \ldots, \|X_{n-1}\|_{\alpha_{n-1}}$ such that, for any $k\geq 1$,
\begin{equation}\label{lemma1simp}
\| \Gamma_{k,X} - \tilde{\Gamma}_{k,X} \|_p \leq C \epsilon \ \ \text{and} \ \ \| \Gamma_{X} - \tilde{\Gamma}_{X} \|_p \leq C \epsilon.
\end{equation}
By the first part of the proof, there exists $k_0 \in \mathbb{N}$ such that for any $k\geq k_0$,
\begin{equation}\label{lemma2simp}
| \tr((\tilde{\Gamma}_X - \tilde{\Gamma}_{k,X})Y)| < \epsilon.
\end{equation}
Hence, by \eqref{Proofineqp}, \eqref{lemma1simp} and \eqref{lemma2simp} we have, for any $k\geq k_0$,
\begin{align*}
& | \tr(\Gamma_{X}Z) - \tr(\Gamma_{k,X}Z) | \\
& \leq | \tr(\Gamma_{X}(Z-Y)) | + | \tr((\Gamma_{X} - \tilde{\Gamma}_{X})Y) | + | \tr((\tilde{\Gamma}_{X} - \tilde{\Gamma}_{k,X})Y) | \\
& \ \ \ + | \tr((\tilde{\Gamma}_{k,X} - \Gamma_{k,X})Y) | + | \tr(\Gamma_{k,X}(Y-Z)) | \\
& \leq \left(\gamma \prod_{i=1}^{n-1} \|X_i\|_{\alpha_i} + C \|Y\|_{\alpha'} + 1 + C \|Y\|_{\alpha'} + C'\prod_{i=1}^{n-1} \|X_i\|_{\alpha_i} \right) \epsilon.
\end{align*}
Since $\|Y\|_{\alpha'} \leq \|Z\|_{\alpha'} + \epsilon$, we proved that $$
\left[ \Gamma^{A_1, \ldots, A_n}(\varphi_k) \right](X_1, \ldots, X_{n-1}) \underset{k \to \infty}{\longrightarrow} \left[ \Gamma^{A_1, \ldots, A_n}(\varphi) \right](X_1, \ldots, X_{n-1})
$$
weakly in $\mathcal{S}^{\alpha}(\mathcal{H})$.

\end{proof}

The next three lemmas give various algebraic properties of multiple operator integrals which will be used in Section $\ref{HOPF}$ and Section $\ref{ProofofMR}$. The proofs of the following results are quite similar: we first prove them in the case $p=2$ for which the $w^*$-continuity of multiple operator integrals allows to reduce the computations to elementary tensors of functions, and then deduce the general case $1 \leq p <\infty$ by approximating the operators in $\mathcal{S}^p(\mathcal{H})$ by operators in $\mathcal{S}^2(\mathcal{H}) \cap \mathcal{S}^p(\mathcal{H})$.

\begin{lemma}\label{simplification1}
Let $1 \leq p < \infty$. Let $n \geq 2$ and $1\leq j \leq n-1$. Let $A_1, \ldots, A_n$ be selfadjoint operators on $\mathcal{H}$. Let $\phi_1 \in L^{\infty}(\lambda_{A_1} \times \cdots \times \lambda_{A_n})$ and $\phi_2 \in L^{\infty}(\lambda_{A_j} \times \lambda_{A_{j+1}})$ be such that
$$
\Gamma^{A_1, \ldots, A_n}(\phi_1) \in \mathcal{B}_{n-1}(\mathcal{S}^p(\mathcal{H})) \ \
\text{and} \ \
\Gamma^{A_j, A_{j+1}}(\phi_2) \in \mathcal{B}(\mathcal{S}^p(\mathcal{H})).
$$
We define $\widetilde{\phi_2} \in L^{\infty}(\lambda_{A_1} \times \cdots \times \lambda_{A_n})$ by
\begin{equation}\label{fortilde}
\widetilde{\phi_2}(x_1, \ldots, x_n) = \phi_2(x_j, x_{j+1})
\end{equation}
a.e. on $\sigma(A_1) \times \cdots \times \sigma(A_n)$.
Then
$$
\Gamma^{A_1, \ldots, A_n}(\phi_1 \widetilde{\phi_2}) \in \mathcal{B}_{n-1}(\mathcal{S}^p(\mathcal{H}))
$$
and for all $K_1, \ldots, K_{n-1} \in \mathcal{S}^p(\mathcal{H})$ we have
\begin{align}\label{fortildeeq}
\begin{split}
& \left[ \Gamma^{A_1, \ldots, A_n}(\phi_1 \widetilde{\phi_2})\right] (K_1, \ldots, K_{n-1}) \\
& \ \ \ \ \ \ \ \ \ = \left[ \Gamma^{A_1, \ldots, A_n}(\phi_1)\right] \left(K_1, \ldots, K_{j-1}, \left[\Gamma^{A_j, A_{j+1}}(\phi_2)\right](K_j), K_{j+1}, \ldots, K_{n-1}\right).
\end{split}
\end{align}
\end{lemma}

\begin{proof}
Assume that $p=2$. We first prove the result when $\phi_1 = f_1 \otimes \cdots \otimes f_n$ and $\phi_2 = g_j \otimes g_{j+1}$ where for any $1\leq i \leq n, f_i\in L^{\infty}(\lambda_{A_i}), g_j \in L^{\infty}(\lambda_{A_j}), g_{j+1} \in L^{\infty}(\lambda_{A_{j+1}})$. In this case,
$$
\phi_1 \widetilde{\phi_2} = f_1 \otimes \cdots \otimes f_{j-1} \otimes f_j g_j \otimes f_{j+1} g_{j+1} \otimes f_{j+2} \otimes \cdots \otimes f_n
$$
so we have, by \eqref{MOItensor},
\begin{align*}
& \left[ \Gamma^{A_1, \ldots, A_n}(\phi_1 \widetilde{\phi_2})\right] (K_1, \ldots, K_{n-1}) \\
& = f_1(A_1)K_1 \ldots K_{j-1} f_j(A_j) g_j(A_j) K_j g_{j+1}(A_{j+1}) f_{j+1}(A_{j+1}) K_{j+1} \ldots K_{n-1}f_n(A_n) \\
& = f_1(A_1)K_1 \ldots K_{j-1} f_j(A_j) \left[\Gamma^{A_j, A_{j+1}}(g_j\otimes g_{j+1})\right](K_j) f_{j+1}(A_{j+1}) K_{j+1} \ldots K_{n-1}f_n(A_n) \\
& = \left[ \Gamma^{A_1, \ldots, A_n}(\phi_1)\right] \left(K_1, \ldots, K_{j-1}, \left[\Gamma^{A_j, A_{j+1}}(\phi_2)\right](K_j), K_{j+1}, \ldots, K_{n-1}\right),
\end{align*}
which proves the result for such $\phi_1$ and $\phi_2$. Note that this formula is bilinear in $(\phi_1, \phi_2)$, hence the result holds true whenever $\phi_1 \in L^{\infty}(\lambda_{A_1}) \otimes \cdots \otimes L^{\infty}(\lambda_{A_n})$ and $\phi_2 \in L^{\infty}(\lambda_{A_j}) \otimes L^{\infty}(\lambda_{A_{j+1}})$.

In the general case, we let $(\phi_{1,s})_{s\in S} \subset L^{\infty}(\lambda_{A_1}) \otimes \cdots \otimes L^{\infty}(\lambda_{A_n})$ and $(\phi_{2,t})_{t\in T} \subset L^{\infty}(\lambda_{A_j}) \otimes L^{\infty}(\lambda_{A_{j+1}})$ be two nets converging to $\phi_1$ and $\phi_2$, respectively for the $w^*$-topology of $L^{\infty}(\lambda_{A_1} \times \cdots \times \lambda_{A_n})$ and for the $w^*$-topology of $L^{\infty}(\lambda_{A_j} \times \lambda_{A_{j+1}})$. Fix $s\in S$ and assume first that $\phi_{1,s} = f_1 \otimes \cdots \otimes f_n$. By the previous computation, we have, for any $t\in T$,
\begin{align}\label{for1}
\begin{split}
& \left[ \Gamma^{A_1, \ldots, A_n}(\phi_{1,s} \widetilde{\phi_{2,t}})\right] (K_1, \ldots, K_{n-1}) \\
& = f_1(A_1)K_1 \ldots K_{j-1} f_j(A_j) \left[\Gamma^{A_j, A_{j+1}}(\phi_{2,t})\right](K_j) f_{j+1}(A_{j+1}) K_{j+1} \ldots K_{n-1}f_n(A_n).
\end{split}
\end{align}
where $\widetilde{\phi_{2,t}}$ is defined as in \eqref{fortilde}. By the $w^*$-continuity of $\Gamma^{A_j, A_{j+1}}$, we get that the right-hand side of \eqref{for1} converges, in the $w^*$-topology of $\mathcal{S}^2(\mathcal{H})$, to
\begin{align*}
& f_1(A_1)K_1 \ldots K_{j-1} f_j(A_j) \left[\Gamma^{A_j, A_{j+1}}(\phi_2)\right](K_j) f_{j+1}(A_{j+1}) K_{j+1} \ldots K_{n-1}f_n(A_n) \\
& = \left[ \Gamma^{A_1, \ldots, A_n}(\phi_{1,s})\right] \left(K_1, \ldots, K_{j-1}, \left[\Gamma^{A_j, A_{j+1}}(\phi_2)\right](K_j), K_{j+1}, \ldots, K_{n-1}\right).
\end{align*}
For the left-hand side of \eqref{for1}, we show that $(\phi_{1,s} \widetilde{\phi_{2,t}})_{t\in T}$ $w^*$-converges to $\phi_{1,s} \widetilde{\phi_2}$. Indeed, let $g\in L^1(\lambda_{A_1} \times \cdots \times \lambda_{A_n})$. Then, writing $\Omega = \sigma(A_1) \times \cdots \times \sigma(A_n)$, we have, by Fubini's theorem,
\begin{align*}
& \int_{\Omega} \phi_{1,s} \widetilde{\phi_{2,t}} ~ g \ \text{d}\lambda_{A_1} \cdots \text{d}\lambda_{A_n} \\
& = \int_{\sigma(A_j) \times \sigma(A_{j+1})} \phi_{2,t} \left( \int_{\prod_{i \neq j,j+1} \sigma(A_i)} \phi_{1,s} g \ \prod_{i \neq j,j+1}\text{d}\lambda_{A_j} \right) \text{d}\lambda_{A_j}\text{d}\lambda_{A_{j+1}} \\
& := \int_{\sigma(A_j) \times \sigma(A_{j+1})} \phi_{2,t} \psi_s \ \text{d}\lambda_{A_j}\text{d}\lambda_{A_{j+1}}.
\end{align*}
By Fubini's theorem, we have the inequality
\begin{align*}
\int_{\sigma(A_j) \times \sigma(A_{j+1})} | \psi_s | \ \text{d}\lambda_{A_j}\text{d}\lambda_{A_{j+1}}
& \leq \int_{\Omega} | \phi_{1,s} g | \ \text{d}\lambda_{A_1} \cdots \text{d}\lambda_{A_n} \\
& \leq \| \phi_{1,s} \|_{\infty} \| g\|_{1},
\end{align*}
which shows that $\psi_s \in L^1(\lambda_{A_j} \times \lambda_{A_{j+1}})$. Hence,
$$
\int_{\sigma(A_j) \times \sigma(A_{j+1})} \phi_{2,t} \psi_s \ \text{d}\lambda_{A_j}\text{d}\lambda_{A_{j+1}} \underset{t}{\longrightarrow} \int_{\sigma(A_j) \times \sigma(A_{j+1})} \phi_2 \psi_s \ \text{d}\lambda_{A_j}\text{d}\lambda_{A_{j+1}},
$$
which is in turn equal to $\displaystyle \int_{\Omega} \phi_{1,s} \widetilde{\phi_2} ~ g \ \text{d}\lambda_{A_1} \cdots \text{d}\lambda_{A_n}$. This shows that $(\phi_{1,s} \widetilde{\phi_{2,t}})_{t\in T}$ $w^*$-converges to $\phi_{1,s} \widetilde{\phi_2}$. By $w^*$-continuity of multiple operator integrals, we have, taking the limit in the weak topology of $\mathcal{S}^2(\mathcal{H})$ in \eqref{for1},
\begin{align}\label{for2}
\begin{split}
& \left[ \Gamma^{A_1, \ldots, A_n}(\phi_{1,s} \widetilde{\phi_2})\right] (K_1, \ldots, K_{n-1}) \\
& = \left[ \Gamma^{A_1, \ldots, A_n}(\phi_{1,s})\right] \left(K_1, \ldots, K_{j-1}, \left[\Gamma^{A_j, A_{j+1}}(\phi_2)\right](K_j), K_{j+1}, \ldots, K_{n-1}\right).
\end{split}
\end{align}
Note that, by linearity, this equality holds true whenever $\phi_{1,s} \in L^{\infty}(\lambda_{A_1}) \otimes \cdots \otimes L^{\infty}(\lambda_{A_n}).$ Since $(\phi_{1,s} \widetilde{\phi_2})_{s\in S}$ $w^*$-converges to $\phi_1 \widetilde{\phi_2}$ we have, by taking the limit in the weak topology of $\mathcal{S}^2(\mathcal{H})$ in \eqref{for2},
\begin{align*}
& \left[ \Gamma^{A_1, \ldots, A_n}(\phi_1 \widetilde{\phi_2})\right] (K_1, \ldots, K_{n-1}) \\
& = \left[ \Gamma^{A_1, \ldots, A_n}(\phi_1)\right] \left(K_1, \ldots, K_{j-1}, \left[\Gamma^{A_j, A_{j+1}}(\phi_2)\right](K_j), K_{j+1}, \ldots, K_{n-1}\right).
\end{align*}

Assume now that $1\leq p < \infty$ and let $K_1, \ldots, K_{n-1} \in \mathcal{S}^2(\mathcal{H}) \cap \mathcal{S}^p(\mathcal{H})$. By assumption, there exist $A_p, B_p >0$ such that 
\begin{align}\label{densitysimp1}
\begin{split}
& \| \left[ \Gamma^{A_1, \ldots, A_n}(\phi_1)\right] \left(K_1, \ldots, K_{j-1}, \left[\Gamma^{A_j, A_{j+1}}(\phi_2)\right](K_j), K_{j+1}, \ldots, K_{n-1}\right) \|_p \\
& \leq A_p \|K_1\|_p \ldots \|K_{j-1}\|_p \| \left[\Gamma^{A_j, A_{j+1}}(\phi_2)\right](K_j) \|_p \|K_{j+1}\|_p \ldots \| K_{n-1}\|_p \\
& \leq A_p B_p \prod_{i=1}^{n-1} \|K_i\|_p.
\end{split}
\end{align}
Since for all $1\leq i \leq n-1, K_i \in \mathcal{S}^2(\mathcal{H})$, equality \eqref{fortildeeq} holds and we deduce the inequality
\begin{align}\label{densitysimp2}
\left\| \left[ \Gamma^{A_1, \ldots, A_n}(\phi_1 \widetilde{\phi_2})\right] (K_1, \ldots, K_{n-1}) \right\|_p  \leq A_p B_p \prod_{i=1}^{n-1} \|K_i\|_p.
\end{align}
By density of $\mathcal{S}^2(\mathcal{H}) \cap \mathcal{S}^p(\mathcal{H})$ in $\mathcal{S}^p(\mathcal{H})$, we get that $\Gamma^{A_1, \ldots, A_n}(\phi_1 \widetilde{\phi_2}) \in \mathcal{B}_{n-1}(\mathcal{S}^p(\mathcal{H}))$ and that inequalities \eqref{densitysimp1} and \eqref{densitysimp2} hold true for any $K_1, \ldots, K_{n-1} \in \mathcal{S}^p(\mathcal{H})$.

Finally, to prove equality \eqref{fortildeeq} in the case when $K_1, \ldots, K_{n-1} \in \mathcal{S}^p(\mathcal{H})$, we approximate $K_i, 1\leq i \leq n-1,$ by elements of $\mathcal{S}^2(\mathcal{H}) \cap \mathcal{S}^p(\mathcal{H})$, using inequalities \eqref{densitysimp1} and \eqref{densitysimp2}.

\end{proof}

\begin{lemma}\label{simplificationsep}
Let $1 \leq p < \infty$. Let $n \geq 3$ and $2\leq j \leq n-1$. Let $A_1, \ldots, A_n$ be selfadjoint operators on $\mathcal{H}$. Let $\phi_1 \in L^{\infty}(\lambda_{A_1} \times \cdots \times \lambda_{A_j})$ and $\phi_2 \in L^{\infty}(\lambda_{A_j} \times \cdots \times \lambda_{A_n})$ be such that
$$
\Gamma^{A_1, \ldots, A_j}(\phi_1) \in \mathcal{B}_{j-1}(\mathcal{S}^p(\mathcal{H})) \ \
\text{and} \ \
\Gamma^{A_j, \ldots, A_n}(\phi_2) \in \mathcal{B}_{n-j}(\mathcal{S}^p(\mathcal{H})).
$$
We define $\phi \in L^{\infty}(\lambda_{A_1} \times \cdots \times \lambda_{A_n})$ by
\begin{equation*}
\phi(x_1,\ldots, x_n) = \phi_1(x_1, \ldots, x_j) \phi_2(x_j, \ldots, x_n)
\end{equation*}
a.e. on $\sigma(A_1) \times \cdots \times \sigma(A_n)$.
Then
$$
\Gamma^{A_1, \ldots, A_n}(\phi) \in \mathcal{B}_{n-1}(\mathcal{S}^p(\mathcal{H}))
$$
and for all $K_1, \ldots, K_{n-1} \in \mathcal{S}^p(\mathcal{H})$ we have
\begin{align}\label{simplificationsepid}
\begin{split}
& \left[ \Gamma^{A_1, \ldots, A_n}(\phi)\right] (K_1, \ldots, K_{n-1}) \\
& \ \ \ \ = \left[ \Gamma^{A_1, \ldots, A_j}(\phi_1)\right] (K_1, \ldots, K_{j-1}) \left[ \Gamma^{A_j, \ldots, A_n}(\phi_2)\right] (K_j, \ldots, K_{n-1}).
\end{split}
\end{align}
\end{lemma}

\begin{proof}
Assume first that $p=2$. In the case when $\phi_1$ and $\phi_2$ are elementary tensors, it is straightforward to check the identity \eqref{simplificationsepid}. In the general case, we let $(\phi_{1,s})_{s\in S} \subset L^{\infty}(\lambda_{A_1}) \otimes \cdots \otimes L^{\infty}(\lambda_{A_j})$ and $(\phi_{2,t})_{t\in T} \subset L^{\infty}(\lambda_{A_j}) \otimes \cdots \otimes L^{\infty}(\lambda_{A_n})$ be two nets converging to $\phi_1$ and $\phi_2$, respectively for the $w^*$-topology of $L^{\infty}(\lambda_{A_1} \times \cdots \times \lambda_{A_j})$ and for the $w^*$-topology of $L^{\infty}(\lambda_{A_j} \times \cdots \times \lambda_{A_n})$. For any $s \in S$ and any $t\in T$, we have
\begin{align}\label{simplificationsepid2}
\begin{split}
& \left[ \Gamma^{A_1, \ldots, A_n}(\phi_{1,s} \phi_{2,t} \right] (K_1, \ldots, K_{n-1}) \\
& \ \ \ \ = \left[ \Gamma^{A_1, \ldots, A_j}(\phi_{1,s})\right] (K_1, \ldots, K_{j-1}) \left[ \Gamma^{A_j, \ldots, A_n}(\phi_{2,t})\right] (K_j, \ldots, K_{n-1}).
\end{split}
\end{align}
For a fixed $s\in S$, $(\phi_{1,s} \phi_{2,t})_{t\in T}$ converges to $\phi_{1,s} \phi_2$ and $(\phi_{1,s} \phi_2)_{s\in S}$ converges to $\phi = \phi_1 \phi_2$ for the $w^*$-topology of $L^{\infty}(\lambda_{A_1} \times \cdots \times \lambda_{A_n})$. Hence by taking the limit on $t\in T$ and then on $s\in S$ in \eqref{simplificationsepid2}, we get \eqref{simplificationsepid}.

Now let $1 \leq p <\infty$ and $K_1, \ldots, K_{n-1} \in \mathcal{S}^2(\mathcal{H}) \cap \mathcal{S}^p(\mathcal{H})$. Then equality \eqref{simplificationsepid} holds and by assumption, there exist $A_p, B_p >0$ such that
\begin{align*}
\begin{split}
& \| \left[ \Gamma^{A_1, \ldots, A_n}(\phi)\right] (K_1, \ldots, K_{n-1})\|_p \\
& \leq \| \left[ \Gamma^{A_1, \ldots, A_j}(\phi_1)\right] (K_1, \ldots, K_{j-1}) \|_p \| \left[ \Gamma^{A_j, \ldots, A_n}(\phi_2)\right] (K_j, \ldots, K_{n-1}) \|_p \\
& \leq A_p \|K_1\|_p \ldots \|K_{j-1}\|_p ~ B_p \|K_j \|_p \ldots \|K_{n+1}\|_p,
\end{split}
\end{align*}
which shows that $\Gamma^{A_1, \ldots, A_n}(\phi) \in \mathcal{B}_{n-1}(\mathcal{S}^p(\mathcal{H}))$. Finally, we deduce \eqref{simplificationsepid} by approximation like in the proof of Lemma $\ref{simplification1}$.

\end{proof}

\begin{lemma}\label{simplification2}
Let $1 \leq p < \infty$. Let $n \geq 2$ and $1\leq j \leq n$. Let $A_1, \ldots, A_n$ be selfadjoint operators in $\mathcal{H}$. Let $\phi \in L^{\infty}(\lambda_{A_1} \times \cdots \times \lambda_{A_{j-1}} \times \lambda_{A_{j+1}} \times \cdots \times \lambda_{A_n})$ and assume, if $n\geq 3$, that
$$
\Gamma^{A_1, \ldots, A_{j-1}, A_{j+1}, \ldots, A_n}(\phi) \in \mathcal{B}_{n-2}(\mathcal{S}^p(\mathcal{H})).
$$
We define $\widetilde{\phi} \in L^{\infty}(\lambda_{A_1} \times \cdots \times \lambda_{A_n})$ by
\begin{equation}\label{fortilde2}
\widetilde{\phi}(x_1, \ldots, x_n) = \phi(x_1, \ldots, x_{j-1}, x_{j+1}, \ldots, x_n)
\end{equation}
a.e. on $\sigma(A_1) \times \cdots \times \sigma(A_n)$. Then
$$
\Gamma^{A_1, \ldots, A_n}(\widetilde{\phi}) \in \mathcal{B}_{n-1}(\mathcal{S}^p(\mathcal{H}))
$$
and for any $K_1, \ldots, K_{n-1} \in \mathcal{S}^p(\mathcal{H})$, we have
\begin{enumerate}
\item[(i)] If $2\leq j \leq n-1$,
\begin{align*}
& \left[ \Gamma^{A_1, \ldots, A_n}(\widetilde{\phi})\right] (K_1, \ldots, K_{n-1}) \\
& \ \ \ \ \ \ \ \ \ = \left[ \Gamma^{A_1, \ldots, A_{j-1}, A_{j+1}, \ldots, A_n}(\phi)\right] \left(K_1, \ldots, K_{j-2}, K_{j-1}K_j, K_{j+1}, \ldots, K_{n-1}\right).
\end{align*}
\item[(ii)] If $j=1$,
\begin{align*}
\left[ \Gamma^{A_1, \ldots, A_n}(\widetilde{\phi})\right] (K_1, \ldots, K_{n-1}) = K_1\left[ \Gamma^{A_2, \ldots, A_n}(\phi)\right] \left(K_2, \ldots, K_{n-1}\right).
\end{align*}
\item[(iii)] If $j=n$,
\begin{align*}
\left[ \Gamma^{A_1, \ldots, A_n}(\widetilde{\phi})\right] (K_1, \ldots, K_{n-1}) = \left[ \Gamma^{A_1, \ldots, A_{n-1}}(\phi)\right] \left(K_1, \ldots, K_{n-2}\right) K_{n-1}.
\end{align*}
\end{enumerate}
\end{lemma}

\begin{proof}
We only prove $(i)$, in the case when $n\geq 3$. The case $n=2$ and the second and third claims can be proved similarly. Assume that $2\leq j \leq n-1$. We first assume that $p=2$. If $\phi = f_1 \otimes \cdots \otimes f_{j-1} \otimes f_{j+1} \otimes \cdots \otimes f_n, f_i \in L^{\infty}(\lambda_{A_i}), 1\leq i\neq j \leq n$, we have
\begin{equation*}
\widetilde{\phi} = f_1 \otimes \cdots \otimes f_{j-1} \otimes 1 \otimes f_{j+1} \otimes \cdots \otimes f_n
\end{equation*}
so that
\begin{align*}
& \left[ \Gamma^{A_1, \ldots, A_n}(\widetilde{\phi})\right] (K_1, \ldots, K_{n-1}) \\
&  = f_1(A_1)K_1 \ldots K_{j-2}f_{j-1}(A_{j-1}) K_{j-1} K_j f_{j+1}(A_{j+1}) K_{j+1} \ldots K_{n-1}f_n(A_n) \\
& = \left[ \Gamma^{A_1, \ldots, A_{j-1}, A_{j+1}, \ldots, A_n}(\phi)\right] \left(K_1, \ldots, K_{j-1}, K_j K_{j+1}, K_{j+1}, \ldots, K_{n-1}\right).
\end{align*}
By linearity, this formula holds true whenever $\phi \in L^{\infty}(\lambda_{A_1}) \otimes \cdots L^{\infty}(\lambda_{A_{j-1}}) \otimes L^{\infty}(\lambda_{A_{j+1}}) \otimes \cdots \otimes L^{\infty}(\lambda_{A_n})$.

In the general case, we let $(\phi_s)_{s\in S} \subset L^{\infty}(\lambda_{A_1}) \otimes \cdots L^{\infty}(\lambda_{A_{j-1}}) \otimes L^{\infty}(\lambda_{A_{j+1}}) \otimes \cdots \otimes L^{\infty}(\lambda_{A_n})$ to be a net converging to $\phi$ for the $w^*$-topology of $L^{\infty}(\lambda_{A_1} \times \cdots \times \lambda_{A_{j-1}} \times \lambda_{A_{j+1}} \times \cdots \times \lambda_{A_n})$. For any $s\in S$, we define $\widetilde{\phi_s}$ as in \eqref{fortilde2}. Then, it is easy to see that $(\widetilde{\phi_s})_{s\in S}$ converges to $\widetilde{\phi}$ for the $w^*$-topology of $L^{\infty}(\lambda_{A_1} \times \cdots \times \lambda_{A_n})$. We conclude using the $w^*$-continuity of multiple operator integral like in the proof of Lemma $\ref{simplification1}$.

In the case when $1\leq p < \infty$, we argue as in the end of the proof of Lemma $\ref{simplification1}$. Details are left to the reader.

\end{proof}

\subsection{Higher order perturbation formula}\label{HOPF}

In this section, we first extend an important result on boundedness of mutiple operator integrals asssociated to divided differences $f^{[n]}$ in the case when $f$ is $n$-times differentiable with bounded $n$th derivative $f^{(n)}$. This will justify that all the operators appearing in the sequel are well-defined. Secondly, we will prove a higher order perturbation formula for differences of multiple operator integrals.\\

Let us recall the definition of the divided differences. Let $f : \mathbb{R} \to \mathbb{C}$ be differentiable. The divided
difference of the first order $f^{[1]}\colon\mathbb{R}^2\to\mathbb{C}$
is defined by
\begin{align*}
{f^{[1]} (x_0,x_1)} :=
\begin{cases}\frac
{ f(x_0) - f(x_1)}{x_0-x_1}, & \text{if}\ x_0
\neq x_1 \\
f'(x_0) & \text{if}\ x_0=x_1
\end{cases}, \qquad x_0, x_1\in\mathbb{R}.
\end{align*}
If $f'$ is bounded then $f^{[1]}$ is a bounded Borel function on $\mathbb{R}^2$ and
if in addition $f'$ is continuous, then $f^{[1]}\in C_b(\mathbb{R}^2)$.

If $n\geq 2$ and $f$ is $n$-times differentiable on $\mathbb{R}$, the divided difference of the $n$th
order $f^{[n]}\colon\mathbb{R}^{n+1}\to\mathbb{C}$
is defined recursively by
\begin{align*}
{f^{[n]} (x_0,x_1,\ldots,x_n)} :=
\begin{cases}\frac
{f^{[n-1]}(x_0,x_2,\ldots,x_n) - f^{[n-1]}(x_1, x_2 \ldots,x_n)}{x_0-x_1}, & \text{if}\ x_0
\neq x_1 \\
\partial_1 f^{[n-1]}(x_1,x_2,\ldots,x_n) & \text{if}\ x_0=x_1
\end{cases},
\end{align*}
for all $x_0, \ldots, x_n \in \mathbb{R}$, where $\partial_i$ stands for the partial derivative with respect to the $i$-th variable. If $f^{(n)}$ is bounded then $f^{[n]}$ is a bounded Borel function on $\mathbb{R}^{n+1}$ and if in addition $f^{(n)}$ is continuous, then $f^{[n]}\in C_b(\mathbb{R}^{n+1})$.

It is well-known that $f^{[n]}$ is symmetric under permutation of its arguments. Therefore, for all $1 \leq i \leq n$ and for all $x_0, \ldots, x_n \in \mathbb{R}$,
\begin{align*}
&f^{[n]} (x_0,x_1,\ldots,x_n) \\
& \ \ \ \ = \dfrac{f^{[n-1]}(x_0,\ldots,x_{i-1}, x_{i+1}, \ldots,x_n) - f^{[n-1]}(x_0,\ldots,x_{i-2}. x_{i},x_{i+1},\ldots,x_n)}{x_{i-1}-x_i}
\end{align*}
 if $x_{i-1} \neq x_i$ and
$$
f^{[n]} (x_0,x_1,\ldots,x_n) = \partial_i f^{[n-1]}(x_1,\ldots,x_n)
$$
if $x_{i-1}=x_i.$

Let $n \in \mathbb{N}, n\geq 1$. For a bounded Borel function $g$ on $\mathbb{R}$, we define, for any $x_0, \ldots, x_n \in \mathbb{R}$,
$$
\varphi_{n,g}(x_0, \ldots, x_n) = \int_{R_n} g\left( \sum_{j=0}^n s_j x_j \right) \text{d}\lambda_n (s_1, \ldots, s_n),
$$
where
$$
R_n = \left\lbrace (s_1,\ldots, s_n) \in \mathbb{R}^n : \ \sum_{j=1}^n s_j \leq 1, s_j \geq 0, 1\leq j \leq n \right\rbrace,
$$
$s_0=1-\sum_{j=1}^n s_j$ and $\lambda_n$ is the Lebesgue measure on $\mathbb{R}^n$.

Let $f$ be $n$-times differentiable on $\mathbb{R}$ with $f^{(n)}$ bounded. Then we have
\begin{equation}\label{DDint}
f^{[n]} = \varphi_{n,f^{(n)}}.
\end{equation}
This follows e.g. from \cite[Formula (7.12)]{DVL}.\\

In the sequel, we will work with selfadjoint operators $A_1, A_2, \ldots, A_n, n\in \mathbb{N}, n\geq 2$. If $\psi : \mathbb{R}^n \rightarrow \mathbb{C}$ is a bounded Borel function, let $\tilde{\psi}$ be the class of the restriction $\psi_{|\sigma(A_1) \times \sigma(A_2) \times \cdots \times \sigma(A_n)}$ in $L^{\infty}\left(\prod_{i=1}^n
\lambda_{A_i}\right)$. Then, we will denote by $\Gamma^{A_1,A_2, \ldots, A_n}(\psi)$ the multiple operator integral $\Gamma^{A_1,A_2, \ldots, A_n}(\tilde{\psi})$.

\vspace{0.5cm}

\begin{theorem}\label{BddnessMOI}
Let $1 < p < \infty$, $n\in\mathbb{N}, n\geq 1$, $f$ be $n$-times differentiable on $\mathbb{R}$ with $f^{(n)}$ bounded. Let $A_1, \ldots, A_{n+1}$ be selfadjoint operators in $\mathcal{H}$. Then $
\Gamma^{A_1,A_2, \ldots, A_{n+1}}(f^{[n]}) \in \mathcal{B}_n(\mathcal{S}^{pn} \times \cdots \times \mathcal{S}^{pn},\mathcal{S}^p)
$
and there exists $c_{p,n} > 0$ depending only on $p$ and $n$ such that for any $X_1,  \ldots, X_n \in \mathcal{S}^{np}(\mathcal{H})$,
\begin{equation}\label{IneqDDn}
\| \left[\Gamma^{A_1,A_2, \ldots, A_{n+1}}(f^{[n]})\right](X_1,\ldots,X_n) \|_p \leq c_{p,n} \|f^{(n)}\|_{\infty}\|X_1\|_{np} \cdots \|X_n\|_{np}.
\end{equation}
In particular $\Gamma^{A_1,A_2, \ldots, A_{n+1}}(f^{[n]}) \in \mathcal{B}_n(\mathcal{S}^p(\mathcal{H}))$ with
\begin{equation}\label{IneqDDn2}
\| \Gamma^{A_1,A_2, \ldots, A_{n+1}}(f^{[n]}) \|_{\mathcal{B}_n(\mathcal{S}^p)} \leq c_{p,n} \|f^{(n)}\|_{\infty}.
\end{equation}
\end{theorem}

\begin{proof}
Define, for any $k\geq 1$, $g_k(t) = k (f^{(n-1)}(t+1/k) - f^{(n-1)}(t)), t\in \mathbb{R}$. Then $(g_k)_{k\geq 1} \subset C(\mathbb{R})$ is pointwise convergent to $f^{(n)}$ and we have the inequality $|g_k| \leq \|f^{(n)}\|_{\infty}$. By \cite[Theorem 5.3]{PSS-SSF}, there exists a constant $c_{p,n} > 0$ depending only on $p$ and $n$ such that, for any $k\geq 1$,
\begin{align}\label{Approxder}
\| \Gamma^{A_1,A_2, \ldots, A_{n+1}}(\varphi_{n,g_k}) \|_{\mathcal{B}_n(\mathcal{S}^{pn} \times \cdots \times \mathcal{S}^{pn},\mathcal{S}^p)} \leq c_{p,n} \|g_k \|_{\infty} \leq c_{p,n} \|f^{(n)}\|_{\infty}.
\end{align}
The proof is given in the case when $A_1 = \cdots = A_{n+1}$ but the arguments from the proof of \cite[Theorem 2.2]{LMS} allow to extend the result in the case when $A_1, \ldots, A_{n+1}$ are distinct.

By Lebesgue's dominated convergence theorem, $\varphi_{n,g_k}$ is pointwise convergent to $\varphi_{n,f^{(n)}}$ on $\mathbb{R}^{n+1}$. Moreover, we have
$$
|\varphi_{n,g_k}| \leq \dfrac{\|g_k\|_{\infty}}{n!} \leq \dfrac{\|f^{(n)}\|_{\infty}}{n!}.
$$
Hence, using Lebesgue's dominated convergence theorem again, we get that $(\varphi_{n,g_k})_{k\geq 1}$ $w^*$-converges to $\varphi_{n,f^{(n)}} = f^{[n]}$ for the $w^*$-topology of $L^{\infty}(\lambda_{A_1} \times \cdots \times \lambda_{A_{n+1}})$. By Lemma $\ref{LemmeUB}$ and \eqref{Approxder} we deduce that
$$
\Gamma^{A_1,A_2, \ldots, A_{n+1}}(f^{[n]}) \in \mathcal{B}_n(\mathcal{S}^{pn} \times \cdots \times \mathcal{S}^{pn},\mathcal{S}^p)
$$
with $\| \Gamma^{A_1,A_2, \ldots, A_{n+1}}(f^{[n]}) \|_{\mathcal{B}_n(\mathcal{S}^{pn} \times \cdots \times \mathcal{S}^{pn},\mathcal{S}^p)} \leq c_{p,n} \|f^{(n)}\|_{\infty}$, from which we deduce inequality \eqref{IneqDDn}. Inequality \eqref{IneqDDn2} follows from the fact that $\|.\|_{pn} \leq \|.\|_p$.

\end{proof}

Let $1<p<\infty$. Let $A, K$ be selfadjoint operators in $\mathcal{H}$ with $K\in \mathcal{S}^p(\mathcal{H})$. A Lipschitz function $f : \mathbb{R} \to \mathbb{C}$ is operator-Lipschitz on $\mathcal{S}^p(\mathcal{H})$ according to \cite[Theorem 1]{PS-Lipschitz} and hence $f(A+K) - f(A) \in \mathcal{S}^p(\mathcal{H}).$ Moreover, we have the formula
$$
f(A+K) - f(A) = \left[ \Gamma^{A+K,A}(f^{[1]})\right](K),
$$
see for instance \cite[Theorem 7.4]{dPSW}.

We will prove a higher order counterpart of this result, which will allow us to express differences of multiple operator integrals of the form
$$
\left[ \Gamma^{A_1, \ldots A_{j-1}, B, A_j, \ldots, A_{n-1}}(f^{[n-1]}) - \Gamma^{A_1, \ldots A_{j-1}, A, A_j, \ldots, A_{n-1}}(f^{[n-1]})\right] (K_1, \ldots, K_{n-1})
$$
as a multiple operator integral associated to $f^{[n]}$, provided that $f^{(n-1)}$ and $f^{(n)}$ are bounded and $B-A \in \mathcal{S}^p(\mathcal{H})$.\\

In order to prove Proposition $\ref{simplification0}$ below, we will need the following fact. Let $B$ be a selfadjoint operator in $\mathcal{H}$. By a well-known result of Weyl-Von Neumann (see \cite[Theorem 38.1]{Conway}), there exist an operator $X \in \mathcal{S}^p(\mathcal{H})$, $(b_n)_n \subset{\mathbb{R}}$ and a Hilbertian basis $(e_n)_n$ of $\mathcal{H}$ such that
$$
B = \sum_{n=1}^{\infty} b_n \left\langle e_n, \cdot \right\rangle e_n + X.
$$
For any $i\geq 1$, we let $P_i$ to be the orthogonal projection onto $\text{Span}\left\lbrace e_l, 1\leq l \leq i \right\rbrace.$ $P_i$ is a finite rank projection and $(P_i)_i$ converges strongly to the identity on $\mathcal{H}$. Moreover, we have
$$
BP_i - P_iB = XP_i - P_iX
$$
which converges to $0$ in $\mathcal{S}^p(\mathcal{H})$ because $X \in \mathcal{S}^p(\mathcal{H})$.

A similar statement holds for unitary operators, and even for normal operators, see \cite{Berg}.\\

Note that the following result was proved in \cite[Lemma 3.10]{LMS} in the case when $f^{(n)}$ is continuous, whose proof consists in approximating $f^{[n]}$ in the particular case $p=2$, and then deducing the result for $1 < p < \infty$ from this case. The formula in the general case below is new. Its proof rests on algebraic properties of divided differences and multiple operator integrals.

\begin{proposition}\label{simplification0}
Let $1 < p < \infty$, $n \in \mathbb{N}, n \geq 2$. Let $A_1, \ldots, A_{n-1}, A, B$ be selfadjoint operators in $\mathcal{H}$ such that $B-A \in \mathcal{S}^p(\mathcal{H})$. Let $f$ be $n$-times differentiable on $\mathbb{R}$ such that $f^{(n-1)}$ and $f^{(n)}$ are bounded. Then, for any $K_1, \ldots, K_{n-1} \in \mathcal{S}^p(\mathcal{H})$ and any $1 \leq j \leq n$ we have
\begin{align*}
& \left[ \Gamma^{A_1, \ldots A_{j-1}, B, A_j, \ldots, A_{n-1}}(f^{[n-1]}) - \Gamma^{A_1, \ldots A_{j-1}, A, A_j, \ldots, A_{n-1}}(f^{[n-1]})\right] (K_1, \ldots, K_{n-1}) \\
& \ = \left[ \Gamma^{A_1, \ldots, A_{j-1}, B, A, A_j, \ldots, A_{n-1}}(f^{[n]})\right](K_1, \ldots, K_{j-1}, B-A, K_j, \ldots K_{n-1}).
\end{align*}
\end{proposition}

\begin{proof}
Let $1 \leq j \leq n$. First note that we have the following equality: for any $(x_0, \ldots, x_n) \in \mathbb{R}^{n+1}$,
\begin{align}\label{eqDDsep}
\begin{split}
& f^{[n]}(x_0, \ldots, x_n).(x_{j-1} - x_j) \\
& \ \ \ = f^{[n-1]}(x_0,\ldots,x_{j-1}, x_{j+1}, \ldots,x_n) - f^{[n-1]}(x_0,\ldots,x_{j-2}, x_{j},x_{j+1},\ldots,x_n).
\end{split}
\end{align}
Let $k\geq 1$. Define $\phi_1 = f^{[n]}$, and for any $(x_0, \ldots, x_n) \in \mathbb{R}^{n+1}$,
$$
\phi_2(x_{j-1},x_j) = (x_{j-1} - x_j)\chi_{[-k,k]}(x_{j-1})\chi_{[-k,k]}(x_{j}),
$$
$$\psi_1(x_0, \ldots, x_n) = f^{[n-1]}(x_0,\ldots,x_{j-1}, x_{j+1}, \ldots,x_n)\chi_{[-k,k]}(x_{j-1})\chi_{[-k,k]}(x_{j})$$ and
$$\psi_2(x_0, \ldots, x_n) = f^{[n-1]}(x_0,\ldots,x_{j-2}, x_{j},x_{j+1},\ldots,x_n)\chi_{[-k,k]}(x_{j-1})\chi_{[-k,k]}(x_{j}).$$
Then $\phi_1, \psi_1, \psi_2 \in L^{\infty}(\lambda_{A_1} \times \cdots \times  \lambda_{A_{j-1}} \times \lambda_B \times \lambda_A \times \lambda_{A_j} \times \cdots \times \lambda_{A_n})$, $\phi_2 \in L^{\infty}(\lambda_B \times \lambda_A)$ and after multiplying equality \eqref{eqDDsep} by $\chi_{[-k,k]}(x_{j-1})\chi_{[-k,k]}(x_{j})$ we obtain
\begin{equation}\label{eqfunc}
\phi_1 \widetilde{\phi_2} = \psi_1 - \psi_2,
\end{equation}
where $\widetilde{\phi_2}$ was defined in \eqref{fortilde}.

Assume first that $2 \leq j \leq n-1$. Let $X, K_1, \ldots, K_{n-1} \in \mathcal{S}^p(\mathcal{H})$. Note that
$$
\left[\Gamma^{B,A}(\chi_{[-k,-k]} \otimes \chi_{[-k,-k]})\right](X) = p_k(B)Xp_k(A)
$$
and
$$
\left[\Gamma^{B,A}(\phi_2)\right](X) = Bp_k(B)Xp_k(A) - p_k(B)Xp_k(A)A,
$$
where $p_k = \chi_{[-k,k]}$.
Denote
$$\Gamma_A = \Gamma^{A_1, \ldots, A_{j-1}, A, A_j, \ldots, A_{n-1}}, \ \Gamma_B = \Gamma^{A_1, \ldots, A_{j-1}, B, A_j, \ldots, A_{n-1}}$$
and
$$\Gamma_{B,A} = \Gamma^{A_1, \ldots, A_{j-1}, B, A, A_j, \ldots, A_{n-1}}.$$
Applying the operator $\left[\Gamma_{B,A}(\cdot)\right](K_1, \ldots, K_{j-1}, X, K_j, \ldots K_{n-1})$ to \eqref{eqfunc} gives, by Lemma $\ref{simplification1}$ and Lemma $\ref{simplification2}$,
\begin{align}\label{eqfunc2}
\begin{split}
& \left[ \Gamma_{B,A}(f^{[n]})\right](K_1, \ldots, K_{j-1}, Bp_k(B)Xp_k(A) - p_k(B)Xp_k(A)A, K_j, \ldots, K_{n-1}) \\
& = \left[ \Gamma_B(f^{[n-1]})\right](K_1, \ldots, K_{j-1}, p_k(B)Xp_k(A) K_j, K_{j+1} \ldots, K_{n-1})\\
& \ \ \ \ - \left[ \Gamma_A(f^{[n-1]})\right](K_1, \ldots, K_{j-2}, p_k(B)Xp_k(A) K_{j-1}, K_j \ldots, K_{n-1}).
\end{split}
\end{align}

Let $(P_i)_k$ be an increasing sequence of finite rank projections converging strongly to the identity and such that
\begin{equation}\label{existpro}
BP_i - P_iB \underset{i \to \infty}{\longrightarrow} 0 \ \text{in} \ \mathcal{S}^p(\mathcal{H}).
\end{equation}
As explained before the statement of the Proposition, such sequence exists. We apply equality \eqref{eqfunc2} to $X=P_i$ and we obtain, for any $i\geq 1$,
\begin{align}\label{eqfunc3}
\begin{split}
& \left[ \Gamma_{B,A}(f^{[n]})\right](K_1, \ldots, K_{j-1}, Bp_k(B)P_ip_k(A) - p_k(B)P_ip_k(A)A, K_j, \ldots, K_{n-1}) \\
& = \left[ \Gamma_B(f^{[n-1]})\right](K_1, \ldots, K_{j-1}, p_k(B)P_ip_k(A) K_j, K_{j+1}, \ldots, K_{n-1})\\
& \ \ \ \ - \left[ \Gamma_A(f^{[n-1]})\right](K_1, \ldots, K_{j-2}, p_k(B)P_ip_k(A)K_{j-1}, K_j \ldots, K_{n-1}).
\end{split}
\end{align}
Note that for any $K \in \mathcal{S}^p(\mathcal{H}), KP_i \to K$ and $P_iK\to K$ in $\mathcal{S}^p(\mathcal{H})$, as $i$ goes to $\infty$. This implies that $p_k(B)P_ip_k(A) K_j \to p_k(B)p_k(A) K_j$ and that $p_k(B)P_ip_k(A)K_{j-1} \to p_k(B)p_k(A)K_{j-1}$ as $i$ goes to $\infty$. By continuity of multiple operator integrals stated in Theorem $\ref{BddnessMOI}$, this implies that the right-hand side of \eqref{eqfunc3} converges in $\mathcal{S}^p(\mathcal{H})$ to
\begin{align*}
&\left[ \Gamma_B (f^{[n-1]})\right](K_1, \ldots, K_{j-1}, p_k(B)p_k(A) K_j, K_{j+1}, \ldots, K_{n-1})\\
& - \left[ \Gamma_A (f^{[n-1]})\right](K_1, \ldots, K_{j-2}, p_k(B)p_k(A)K_{j-1}, K_j \ldots, K_{n-1}).
\end{align*}
Using the identity
$$
Bp_k(B)P_ip_k(A) - p_k(B)P_ip_k(A)A = p_k(B)(BP_i-P_iB)p_k(A) + p_k(B)P_i(B-A)p_k(A),
$$
we have, by \eqref{existpro}, that
$$
Bp_k(B)P_ip_k(A) - p_k(B)P_ip_k(A)A \to p_k(B)(B-A)p_k(A)
$$
in $\mathcal{S}^p(\mathcal{H})$, as $i$ goes to $\infty$. Hence, the left-hand side of \eqref{eqfunc3} converges in $\mathcal{S}^p(\mathcal{H})$ to
$$
\left[ \Gamma_{B,A}(f^{[n]})\right](K_1, \ldots, K_{j-1}, p_k(B)(B-A)p_k(A), K_j, \ldots, K_{n-1})
$$
and we proved that
\begin{align}\label{eqfunc4}
\begin{split}
& \left[ \Gamma_{B,A}(f^{[n]})\right](K_1, \ldots, K_{j-1}, p_k(B)(B-A)p_k(A), K_j, \ldots, K_{n-1}) \\
& = \left[ \Gamma_B(f^{[n-1]})\right](K_1, \ldots, K_{j-1}, p_k(B)p_k(A) K_j, K_{j+1}, \ldots, K_{n-1})\\
& \ \ \ \ - \left[ \Gamma_A(f^{[n-1]})\right](K_1, \ldots, K_{j-2}, p_k(B)p_k(A)K_{j-1}, K_j \ldots, K_{n-1}).
\end{split}
\end{align}
Finally, note that $(p_k(B))_{k\geq 1}$ and $(p_k(A))_{k\geq 1}$ converge strongly to the identity as $k$ goes to $\infty$ so $p_k(B)p_k(A) K_j \to K_j$ and $p_k(B)p_k(A)K_{j-1} \to K_{j-1}$ in $\mathcal{S}^p(\mathcal{H})$, as $k$ goes to $\infty$. By assumption, $B-A \in \mathcal{S}^p(\mathcal{H})$ so we have $p_k(B)(B-A)p_k(A) \to B-A$ as $k$ goes to $\infty$. Hence, taking the limit on $k$ in \eqref{eqfunc4} concludes the proof in the case when $2 \leq j \leq n-1$.\\

In the case when $j=1$, the right-hand side of \eqref{eqfunc2} is replaced by
\begin{align*}
& \left[ \Gamma_B (f^{[n-1]})\right](p_k(B)Xp_k(A)K_1, K_2, \ldots, K_{n-1})\\
& \ \ \ \ - p_k(B)Xp_k(A)\left[ \Gamma_A (f^{[n-1]})\right](K_1,\ldots, K_{n-1})
\end{align*}
and when $j=n$, the right-hand side is replaced by
\begin{align*}
& \left[ \Gamma_B (f^{[n-1]})\right](K_1, \ldots, K_{n-1})p_k(B)Xp_k(A)\\
& \ \ \ \ - \left[ \Gamma_A (f^{[n-1]})\right](K_1,\ldots, K_{n-2}, p_k(B)Xp_k(A)K_{n-1}).
\end{align*}
We then apply the same reasonning as before to obtain the result.

\end{proof}

\begin{remark}\label{approximatebdd}
In the latter, we used the projections $p_k$ to approximate the (possibly) unbounded operators $A$ and $B$ by bounded operators. In the case when $A$ and $B$ are bounded selfadjoint operators (without any assumption on the difference $B-A$), the latter proof shows that we have, for any $2\leq j \leq n-1$ and any $X\in \mathcal{S}^p(\mathcal{H})$,
\begin{align*}
& \left[ \Gamma^{A_1, \ldots, A_{j-1}, B, A, A_j, \ldots, A_{n-1}}(f^{[n]})\right](K_1, \ldots, K_{j-1}, BX - XA, K_j, \ldots, K_{n-1}) \\
& = \left[ \Gamma^{A_1, \ldots, A_{j-1}, B, A_j, \ldots, A_{n-1}}(f^{[n-1]})\right](K_1, \ldots, K_{j-1}, X K_j, K_{j+1} \ldots, K_{n-1})\\
& \ \ \ \ - \left[ \Gamma^{A_1, \ldots, A_{j-1}, A, A_j, \ldots, A_{n-1}}(f^{[n-1]})\right](K_1, \ldots, K_{j-2}, X K_{j-1}, K_j \ldots, K_{n-1}).
\end{align*}
When $j=1$, we have
\begin{align*}
& \left[ \Gamma^{A_1, \ldots, A_{j-1}, B, A, A_j, \ldots, A_{n-1}}(f^{[n]})\right](BX - XA, K_1, \ldots, K_{n-1}) \\
& = \left[ \Gamma^{A_1, \ldots, A_{j-1}, B, A_j, \ldots, A_{n-1}}(f^{[n-1]})\right](XK_1, K_2, \ldots, K_{n-1})\\
& \ \ \ \ - X\left[ \Gamma^{A_1, \ldots, A_{j-1}, A, A_j, \ldots, A_{n-1}}(f^{[n-1]})\right](K_1,\ldots, K_{n-1})
\end{align*}
and when $j=n$, we have
\begin{align*}
& \left[ \Gamma^{A_1, \ldots, A_{j-1}, B, A, A_j, \ldots, A_{n-1}}(f^{[n]})\right](K_1, \ldots, K_{n-1},BX - XA) \\
& = \left[ \Gamma^{A_1, \ldots, A_{j-1}, B, A_j, \ldots, A_{n-1}}(f^{[n-1]})\right](K_1, \ldots, K_{n-1})X\\
& \ \ \ \ - \left[ \Gamma^{A_1, \ldots, A_{j-1}, A, A_j, \ldots, A_{n-1}}(f^{[n-1]})\right](K_1,\ldots, K_{n-2}, X K_{n-1}).
\end{align*}
\end{remark}

\section{Differentiability of $t \mapsto f(A+tK) - f(A)$ in $\mathcal{S}^p(\mathcal{H})$}\label{Perturbation}

\subsection{Statements of the main results}\label{SecMR}

In this subsection, we state our main results on $\mathcal{S}^p$-differentiability of functions of operators.

The following generalizes the analogous result of \cite[Theorem 3.7 (ii)]{LMS} from $n$-times continuously differentiable $f$ to $n$-times differentiable functions $f$, with a proof of a completely different nature. It is also the $n$th order analogue of \cite[7.13]{KPSS}.

\begin{theorem}\label{FormulaSA}
Let $1 < p < \infty$, let $A$ and $K$ be bounded selfadjoint operators in $\mathcal{H}$ with $K\in \mathcal{S}^p(\mathcal{H})$. Let $n \in\N, n\geq 1$ and let $f$ be $n$-times differentiable on $\mathbb{R}$ such that $f^{(n)}$ is bounded. Consider the function
$$\varphi : t\in \mathbb{R} \mapsto f(A+tK) - f(A) \in \mathcal{S}^p(\mathcal{H}).$$
Then the function $\varphi$ belongs to $D^n(\mathbb{R},\Sc^p(\H))$ and for every integer $1 \leq k \leq n$,
\begin{equation}\label{diffenrential}
\dfrac{1}{k!} \varphi^{(k)}(t) = \left[ \Gamma^{A+tK, A+tK, \ldots, A+tK}(f^{[k]}) \right] (K, \ldots, K), t\in \mathbb{R}.
\end{equation}
In particular, for any $1\leq k \leq n-1$, $\varphi^{(k)}$ is bounded on any bounded interval of $\mathbb{R}$ and $\varphi^{(n)}$ is bounded on $\mathbb{R}$.
\end{theorem}

We have the same result for unbounded operators, provided that the derivatives of $f$ are bounded, to ensure the boundedness of multiple operator integrals.

\begin{theorem}\label{FormulaSAunbdd}
Let $1 < p < \infty$, $A$ and $K$ be selfadjoint operators in $\mathcal{H}$ with $K\in \mathcal{S}^p(\mathcal{H})$. Let $n \in\N, n\geq 1$ and let $f$ be $n$-times differentiable on $\mathbb{R}$ such that $f^{(i)}$ is bounded for all $1\leq i \leq n$. Consider the function
\begin{equation}\label{defofvarphi}
\varphi : t\in \mathbb{R} \mapsto f(A+tK) - f(A) \in \mathcal{S}^p(\mathcal{H}).
\end{equation}
Then $\varphi$ belongs to $D^n(\mathbb{R},\Sc^p(\H))$ and for every integer $1 \leq k \leq n$, $\varphi^{(k)}$ is bounded on $\mathbb{R}$ and given by
\begin{equation}\label{diffenrentialunbdd}
\dfrac{1}{k!} \varphi^{(k)}(t) = \left[ \Gamma^{A+tK, A+tK, \ldots, A+tK}(f^{[k]}) \right] (K, \ldots, K), t\in \mathbb{R}.
\end{equation}
\end{theorem}

The following allows to express operator Taylor remainders as multiple operator integrals and deduce an $\mathcal{S}^p$-estimate in the case when $f$ has a bounded $n$th derivative. It generalizes \cite[Theorem 3.8]{LMS} where such representation and estimate were obtained for $n$-times continuously differentiable functions $f$.

\begin{proposition}\label{Taylorrmd}
Let $1 < p < \infty$, $n \in\N, n\geq 2$, $A$ and $K$ be selfadjoint operators in $\mathcal{H}$ with $K\in \mathcal{S}^{np}(\mathcal{H})$. Let $f$ be $n$-times differentiable on $\mathbb{R}$ such that $f^{(n)}$ is bounded. Assume that either $A$ is bounded or $f^{(i)}$ is bounded for all $1\leq i \leq n$. Denote
$$
R_{n,p,A,K,f} = f(A+K)-f(A) - \sum_{k=1}^{n-1} \dfrac{1}{k!} \frac{d^k}{dt^k}\Big(f(A+tK)\Big)\Big|_{t=0}.
$$
Then,
\begin{equation}\label{Taylor}
R_{n,p,A,K,f} = \left[ \Gamma^{A+K, A, \ldots, A}(f^{[n]}) \right] (K, \ldots, K),
\end{equation}
and we have the inequality
\begin{equation}\label{Taylorineq}
\|R_{n,p,A,K,f}\|_p \leq c_{p,n} \|f^{(n)}\|_{\infty} \|K\|_{np}^n.
\end{equation}
\end{proposition}

Finally, the result stated below is the $\mathcal{S}^p$-analogue of \cite[Theorem 4.1]{CLSS}. Note that \cite[Theorem 3.7 (ii)]{LMS} establishes the existence of the $n$th derivative of $\varphi$ under the assumptions of Proposition $\ref{Contcase}$. We prove here that $\varphi$ is actually $n$-times continuously differentiable.

\begin{proposition}\label{Contcase}
Let $1 < p < \infty$, let $A$ and $K$ be selfadjoint operators in $\mathcal{H}$ with $K\in \mathcal{S}^p(\mathcal{H})$. Let $n \in\N$ and $f\in C^n(\mathbb{R})$.  Assume that either $A$ is bounded or $f^{(i)}$ is bounded for all $1\leq i\leq n$. Consider the function
$$\varphi : t\in \mathbb{R} \mapsto f(A+tK) - f(A) \in \mathcal{S}^p(\mathcal{H}).$$
Then $\varphi$ belongs to $C^n(\mathbb{R},\Sc^p(\H ))$ and for every integer $1 \leq k \leq n$ and $t\in \mathbb{R}$,
\begin{equation*}
\dfrac{1}{k!} \varphi^{(k)}(t) = \left[ \Gamma^{A+tK, A+tK, \ldots, A+tK}(f^{[k]}) \right] (K, \ldots, K).
\end{equation*}
\end{proposition}

\subsection{Auxiliary lemmas}\label{SubAux}

In this subsection, we will prove important technical lemmas that will be used in Section $\ref{ProofofMR}$.

\begin{lemma}\label{MOIcomm}
Let $1 < p < \infty$, $n \in \mathbb{N}, n \geq 1$. Let $A \in \mathcal{B}(\mathcal{H})$ be a selfadjoint operator and let $Z_1, \ldots, Z_n \in \mathcal{S}^p(\mathcal{H})$ be such that $A$ and $Z_i$ commute, for every $1\leq i\leq n$. Let $f$ be $n$-times differentiable on $\mathbb{R}$ such that $f^{(n)}$ is bounded. Then
$$
\left[ \Gamma^{A, \ldots, A}(f^{[n]}) \right] (Z_1, \ldots, Z_n) = \dfrac{1}{n!} f^{(n)}(A)Z_1\ldots Z_n.
$$
\end{lemma}

\begin{proof}
In this proof, we will use the notation introduced before the statement of Theorem $\ref{BddnessMOI}$. For any $k\geq 1$, we let $\psi_k := \varphi_{n,g_k}$ to be the function defined as in the proof of Theorem $\ref{BddnessMOI}$. For any bounded Borel function $g$, we let $\tilde{g}$ to be the function defined on $\mathbb{R}$ by $\tilde{g}_k(x) = g(x, \ldots, x), x\in \mathbb{R}$. Let us prove first that for any $k\geq 1$,
$$
\left[ \Gamma^{A, \ldots, A}(\psi_k) \right] (Z_1, \ldots, Z_n) = \tilde{\psi}_k(A)Z_1\ldots Z_n.
$$
Fix $k\geq 1$. $A$ is bounded so $\sigma(A) \subset \mathbb{R}$ is bounded and by definition,
$$\Gamma^{A, \ldots, A}(\psi_k) = \Gamma^{A, \ldots, A}(\phi_k)$$
where $\phi_k$ is the class in $L^{\infty}(\lambda_A \times \cdots \times \lambda_A)$ of the restriction of $\psi_k$ to $\sigma(A)^{n+1}$. $g_k$ is continuous on the compact $I= \text{conv}(\sigma(A))$ so there exists a sequence $(P^k_j)_{j\geq 1}$ of polynomial functions converging uniformly to $g_k$ on $I$. For any $j\geq 1$, define $Q^k_j = \varphi_{n, P^k_j}$. It is easy to see that $(Q^k_j)_{j\geq 1}$ converges uniformly to $\psi_k$ on $\sigma(A)^{n+1}$. According to \eqref{DDint}, $Q^k_j = (R^k_j)^{[n]}$ where $R^k_j$ is a polynomial function on $\mathbb{R}$ such that $(R^k_j)^{(n)} = P^k_j$. Hence $Q^k_j$ is a $(n+1)$-variable polynomial function, and in particular, $Q^k_j \in \text{Bor}(\mathbb{R}) \otimes \cdots \otimes \text{Bor}(\mathbb{R})$. Note that for an elementary tensor $g = g_1 \otimes \cdots \otimes g_{n+1} \in \text{Bor}(\mathbb{R}) \otimes \cdots \otimes \text{Bor}(\mathbb{R})$, we have
\begin{align*}
\left[ \Gamma^{A, \ldots, A}(g) \right] (Z_1, \ldots, Z_n)
& = g_1(A)Z_1g_2(A)\ldots g_n(A)Z_ng_{n+1}(A) \\
& = g_1(A)\ldots g_{n+1}(A) Z_1 \ldots Z_n \\
& = \tilde{g}(A) Z_1 \ldots Z_n.
\end{align*}
By linearity, this implies that for any $j \geq 1$,
\begin{equation}\label{eqcomm0}
\left[ \Gamma^{A, \ldots, A}(Q^k_j) \right] (Z_1, \ldots, Z_n) = \tilde{Q}^k_j(A)Z_1\ldots Z_n.
\end{equation}
For any $j\geq 1$, we let $v^k_j \in \text{Bor}(\mathbb{R})$ be such that $v^k_j = P^k_j$ on $I$ and  $v^k_j \to g_k$ uniformly on $\mathbb{R}$. Then
$$
\Gamma^{A, \ldots, A}(\varphi_{n, P^k_j}) = \Gamma^{A, \ldots, A}(\varphi_{n, v^k_j})
$$
and by Theorem $\ref{BddnessMOI}$, there exists a constant $c_{p,n}$ such that
\begin{align*}
& \| \left[ \Gamma^{A, \ldots, A}(Q^k_j) \right] (Z_1, \ldots, Z_n) - \left[ \Gamma^{A, \ldots, A}(\psi_k) \right] (Z_1, \ldots, Z_n) \|_p \\
& = \| \left[ \Gamma^{A, \ldots, A}(\varphi_{n, v^k_j} - \varphi_{n, g_k}) \right] (Z_1, \ldots, Z_n) \|_p \\
& \leq c_{p,n} \| v^k_j - g_k\|_{\infty} \|Z_1\|_p \cdots \|Z_n\|_p \\
& \underset{j\to \infty}{\longrightarrow} 0.
\end{align*}
Note that $(\tilde{Q}^k_j)_{j\geq 1}$ converges uniformly to $\tilde{\phi}_k$ on $\sigma(A)$. Hence, $\tilde{Q}^k_j(A)$ converges to $\tilde{\psi}_k(A)$ in $\mathcal{B}(\mathcal{H})$ so that the right-hand side of \eqref{eqcomm0} converges in $\mathcal{S}^p(\mathcal{H})$ to $\tilde{\psi}_k(A)Z_1\ldots Z_n$. By taking the limit on $j$ in \eqref{eqcomm0} we get
\begin{equation}\label{eqcomm}
\left[ \Gamma^{A, \ldots, A}(\psi_k) \right] (Z_1, \ldots, Z_n) = \tilde{\psi}_k(A)Z_1\ldots Z_n.
\end{equation}
Recall that, from the proof of Theorem $\ref{BddnessMOI}$, the sequence $(\psi_k)_{k\geq 1}$ $w^*$-converges to $f^{[n]}$ for the $w^*$-topology of $L^{\infty}(\lambda_A \times \cdots \times \lambda_A)$ and that $(\Gamma^{A, \ldots, A}(\psi_k))_{k\geq 1} \subset \mathcal{B}_n(\mathcal{S}^p(\mathcal{H}))$ is bounded. Hence, by Lemma $\ref{LemmeUB}$,
$$
\left[ \Gamma^{A, \ldots, A}(\psi_k) \right] (Z_1, \ldots, Z_n) \underset{k \to \infty}{\longrightarrow} \left[ \Gamma^{A, \ldots, A}(f^{[n]}) \right] (Z_1, \ldots, Z_n)
$$
weakly in $\mathcal{S}^p(\mathcal{H})$. On the other hand, $(\tilde{\psi}_k)_{k\geq 1}$ is bounded and is pointwise convergent to $\tilde{f}^{[n]}=\dfrac{1}{n!}f^{(n)}$ so $(\tilde{\psi}_k(A))_{k\geq 1}$ converges strongly to $\dfrac{1}{n!} f^{(n)}(A)$ (see e.g. the proof of \cite[Proposition 3.1]{CLSS}). This implies that the right-hand side of \eqref{eqcomm} converges in $\mathcal{S}^p(\mathcal{H})$ to $\dfrac{1}{n!} f^{(n)}(A)Z_1\ldots Z_n.$ We conclude the proof by taking the limit on $k$ in \eqref{eqcomm}, in the weak topology of $\mathcal{S}^p(\mathcal{H})$.

\end{proof}

From now on, we will adopt the following notation: if $X$ is an operator on $\mathcal{H}$, then for any integer $k$, we denote by $(X)^{k}$ the tuple consisting of $k$ copies of $X$.

\begin{lemma}\label{simplification5}
Let $1 < p < \infty$, $n\in\mathbb{N}, n\geq 2$. Let $A, K$ be selfadjoint operators in $\mathcal{H}$ with $K\in \mathcal{S}^p(\mathcal{H})$ and let $X_1, \ldots, X_{n-1} \in \mathcal{S}^p(\mathcal{H})$. Let $f$ be $n$-times differentiable on $\mathbb{R}$ such that $f^{(n)}$ is bounded. Assume that either $A$ is bounded or $f^{(n-1)}$ is bounded. Let $1 \leq j \leq n$. Define $\psi : t\in \mathbb{R} \to \mathcal{S}^p(\mathcal{H})$ by
$$
\psi(t) = \left[ \Gamma^{(A+tK)^j, (A)^{n-j}}(f^{[n-1]}) \right](X_1, \ldots, X_{n-1}), t\in \mathbb{R}.
$$
Then we have, for any $t\in \mathbb{R}$,
\begin{align}\label{simplification5eq}
\begin{split}
& \psi(t) - \psi(0) \\
& \ \ \ \ = t \sum_{k=1}^{j} \left[ \Gamma^{(A+tK)^{j-k+1},(A)^{n-j+k}}(f^{[n]}) \right](X_1, \ldots, X_{j-k}, K, X_{j-k+1}, \ldots, X_{n-1}).
\end{split}
\end{align}
In particular, $\psi$ is continuous in $0$.
\end{lemma}

\begin{proof}
We have the following decomposition
\begin{align*}
&\psi(t) - \psi(0) \\
& = \sum_{k=1}^j \left[ \Gamma^{(A+tK)^{j-k+1},(A)^{n-j+k-1}}(f^{[n-1]}) - \Gamma^{(A+tK)^{j-k},(A)^{n-j+k}}(f^{[n-1]}) \right](X_1, \ldots, X_{n-1}).
\end{align*}
For any $1\leq k\leq j$, we have, by Proposition $\ref{simplification0}$,
\begin{align*}
& \left[ \Gamma^{(A+tK)^{j-k+1},(A)^{n-j+k-1}}(f^{[n-1]}) - \Gamma^{(A+tK)^{j-k},(A)^{n-j+k}}(f^{[n-1]}) \right](X_1, \ldots, X_{n-1}) \\
& = t \left[ \Gamma^{(A+tK)^{j-k+1},(A)^{n-j+k}}(f^{[n]}) \right](X_1, \ldots, X_{j-k}, K, X_{j-k+1}, \ldots, X_{n-1}),
\end{align*}
from which we deduce \eqref{simplification5eq}.

For the continuity of $\psi$ in $0$, note that by Theorem $\ref{BddnessMOI}$ there exists a constant $c_{p,n} > 0$ such that
\begin{align*}
& \| \psi(t) - \psi(0) \|_p \\
& \leq |t| \sum_{k=1}^{j} \| \left[ \Gamma^{(A+tK)^{j-k+1},(A)^{n-j+k}}(f^{[n]}) \right](X_1, \ldots, X_{j-k}, K, X_{j-k+1}, \ldots, X_{n-1}) \|_p \\
& \leq |t| j c_{p,n} \|f^{(n)} \|_{\infty} \|K\|_p \prod_{i=1}^{n-1} \| X_i \|_p,
\end{align*}
which converges to $0$ as $t$ goes to $0$.

\end{proof}

\begin{lemma}\label{simplification3}
Let $1 < p < \infty$, $n \in\N, n\geq 2$. Let $A$ be a selfadjoint operator in $\mathcal{H}$ and $f$ be $n$-times differentiable on $\mathbb{R}$ such that $f^{(n)}$ is bounded. Assume that either $A$ is bounded or $f^{(n-1)}$ is bounded. Let $X_0 \subset (\mathcal{S}^p(\mathcal{H}))_{\text{sa}}$ be a dense subset. Assume that for any $K_0 \in X_0$, the map $\psi_0 : t\in \mathbb{R} \to \mathcal{S}^p(\mathcal{H})$ defined by
$$
\psi_0(t) = \left[ \Gamma^{A+tK_0,\ldots,A+tK_0}(f^{[n-1]}) \right](K_0, \ldots, K_0), t\in \mathbb{R},
$$
is differentiable in $0$ with
$$
\psi_0'(0) = n \left[ \Gamma^{A,\ldots,A}(f^{[n]}) \right](K_0, \ldots, K_0).
$$
Let $K \in \mathcal{S}^p(\mathcal{H})$ selfadjoint and define $\psi : t\in \mathbb{R} \to \mathcal{S}^p(\mathcal{H})$ by
$$
\psi(t) = \left[ \Gamma^{A+tK,\ldots,A+tK}(f^{[n-1]}) \right](K, \ldots, K), t\in \mathbb{R}.
$$
Then $\psi$ is differentiable in $0$ and
$$
\psi'(0) = n \left[ \Gamma^{A,\ldots,A}(f^{[n]}) \right](K, \ldots, K).
$$
\end{lemma}

\begin{proof}
Let $K\in \mathcal{S}^p(\mathcal{H})$ selfadjoint and show that $\psi$ is differentiable in $0$ with
$$
\psi'(0) = n \left[\Gamma^{A, \ldots A}(f^{[n]}) \right](K, \ldots, K).
$$
Let $\epsilon > 0$ and choose $K_0 \in X_0$ such that $\|K-K_0\|_p \leq \epsilon$. By assumption, $\psi_0$ is differentiable in $0$ and
$$
\psi_0'(0) = n \left[\Gamma^{A, \ldots A}(f^{[n]}) \right](K_0, \ldots, K_0).
$$
Hence, there exists $\mu > 0$ such that for any $|t| < \mu$,
\begin{equation}\label{psiepsilon0}
\|\psi_0(t) - \psi_0(0) - n t \left[\Gamma^{A, \ldots A}(f^{[n]}) \right](K_0, \ldots, K_0) \|_p \leq |t| \epsilon.
\end{equation}
Define $\tilde{\psi} : \mathbb{R} \to \mathcal{S}^p(\mathcal{H})$ by
$$
\tilde{\psi}(t) = \left[\Gamma^{A+tK_0, \ldots A+tK_0}(f^{[n-1]}) \right](K, \ldots, K), t\in \mathbb{R}.
$$
By Lemma $\ref{simplification5}$ we have
\begin{align*}
\tilde{\psi}(t) - \tilde{\psi}(0)
& = t \sum_{k=1}^n \left[ \Gamma^{(A+tK_0)^{n-k+1},(A)^k}(f^{[n]}) \right]((K)^{n-k}, K_0, (K)^{k-1}) \\
& = t \sum_{k=1}^n \left[ \Gamma^{(A+tK_0)^{n-k+1},(A)^k}(f^{[n]}) \right](K_0, \ldots, K_0) + t\psi_{\epsilon}(t) \\
& = \psi_0(t) - \psi_0(0) + t\psi_{\epsilon}(t)
\end{align*}
where
$$
\psi_{\epsilon}(t) = \sum_{k=1}^n \left( \Gamma_{A,K_0,k,t}((K)^{n-k}, K_0, (K)^{k-1}) - \Gamma_{A,K_0,k,t}(K_0,\ldots,K_0) \right).
$$
with $\Gamma_{A,K_0,k,t} = \Gamma^{(A+tK_0)^{n-k+1},(A)^k}(f^{[n]})$.
By Remark $\ref{Ineqmulti}$ and Theorem $\ref{BddnessMOI}$, there exists a constant $\alpha > 0$ depending only on $p, n, \|f^{(n)}\|_{\infty}$ and $\|K\|_p$ such that for any $1 \leq k \leq n$ and any $t\in \mathbb{R}$, 
$$
\| \Gamma_{A,K_0,k,t}((K)^{n-k}, K_0, (K)^{k-1}) - \Gamma_{A,K_0,k,t}(K_0, \ldots, K_0) \|_p \leq \alpha \epsilon
$$
so that we have the estimate $\| \psi_{\epsilon}(t) \| \leq n \alpha \epsilon.$ By the estimate \eqref{psiepsilon0} and triangle inequality, we deduce that for any $|t| < \mu$,
\begin{align}\label{psiepsilon4}
\begin{split}
& \| \tilde{\psi}(t) - \tilde{\psi}(0) - n t \left[\Gamma^{A, \ldots A}(f^{[n]}) \right](K_0, \ldots, K_0) \|_p \\
& \leq \| \psi_0(t) - \psi_0(0)  - n t \left[\Gamma^{A, \ldots A}(f^{[n]}) \right](K_0, \ldots, K_0) \|_p + |t| \| \psi_{\epsilon}(t) \|_p \\
& \leq |t| \epsilon (n\alpha + 1).
\end{split}
\end{align}

By Lemma $\ref{simplification5}$ we have
\begin{align*}
\psi(t) - \tilde{\psi}(t) = t \sum_{k=1}^{n} \left[ \Gamma^{(A+tK)^{n-k+1}, (A+tK_0)^k}(f^{[n]})\right]((K)^{n-k}, K-K_0, (K)^{k-1}).
\end{align*}
Hence, by Remark $\ref{Ineqmulti}$ and Theorem $\ref{BddnessMOI}$, there exists a constant $\beta>0$ depending only on $p, n, \|f^{(n)}\|_{\infty}$ and $\|K\|_p$ such that, for any $t\in \mathbb{R}$
\begin{equation}\label{psiepsilon2}
\| \psi(t) - \tilde{\psi}(t) \|_p \leq n \beta |t| \epsilon.
\end{equation}

Let also $\gamma>0$ be a constant depending only on $p, n, \|f^{(n)}\|_{\infty}$ and $\|K\|_p$ such that
\begin{equation}\label{psiepsilon3}
\| \left[\Gamma^{A, \ldots A}(f^{[n]}) \right](K, \ldots, K) - \left[\Gamma^{A, \ldots A}(f^{[n]}) \right](K_0, \ldots, K_0) \|_p \leq \gamma \epsilon.
\end{equation}

Finally, by triangle inequality and noting that $\psi(0) = \tilde{\psi}(0)$ we have, by \eqref{psiepsilon2}, \eqref{psiepsilon4} and \eqref{psiepsilon3}, 
\begin{align*}
& \| \psi(t) - \psi(0) - n t \left[\Gamma^{A, \ldots A}(f^{[n]}) \right](K, \ldots, K) \|_p \\
& \leq \| \psi(t) - \tilde{\psi}(t)\|_p \\
& \ \ \ \ + \| \tilde{\psi}(t) - \tilde{\psi}(0) - n t \left[\Gamma^{A, \ldots A}(f^{[n]}) \right](K_0, \ldots, K_0)\|_p \\
& \ \ \ \ + \| n t \left[\Gamma^{A, \ldots A}(f^{[n]}) \right](K_0, \ldots, K_0) - nt \left[\Gamma^{A, \ldots A}(f^{[n]}) \right](K, \ldots, K) \|_p \\
& \leq |t| \epsilon (n\beta + n\alpha + n\gamma  + 1),
\end{align*}
for any $|t| < \mu$. This concludes the proof.

\end{proof}

\begin{lemma}\label{simplification4}
Let $1 < p < \infty$, let $A \in \mathcal{B}(\mathcal{H})$ be a selfadjoint operator on $\mathcal{H}$ and let $K=K^* \in \mathcal{S}^p(\mathcal{H})$. Let $n\geq 2$ and let $f$ be $n$-times differentiable on $\mathbb{R}$ such that $f^{(n)}$ is bounded. Let $\nu : \mathbb{R} \to (\mathcal{B}(\mathcal{H}))_{\text{sa}}$ be such that $\nu(0)=A$ and $\nu$ is $\mathcal{S}^p$-differentiable in $0$ with $\nu'(0)=K$. Define $\psi, \tilde{\psi} : \mathbb{R} \to \mathcal{S}^p(\mathcal{H})$ by
$$
\psi(t) = \left[ \Gamma^{A+tK,\ldots,A+tK}(f^{[n-1]}) \right](K, \ldots, K), t\in \mathbb{R},
$$
and
$$
\tilde{\psi}(t) = \left[ \Gamma^{\nu(t),\ldots,\nu(t)}(f^{[n-1]}) \right](K, \ldots, K), t\in \mathbb{R}.
$$
If $\tilde{\psi}$ is differentiable in $0$, then $\psi$ is also differentiable in $0$ and $\psi'(0) = \tilde{\psi}'(0)$.
\end{lemma}

\begin{proof}
Let $\epsilon > 0$. By assumption, there exists $\mu_1 > 0$ such that for any $|t| < \mu_1$,
\begin{equation}\label{nu1}
\| \tilde{\psi}(t) - \tilde{\psi}(0) - t \tilde{\psi}'(0) \|_p \leq |t| \epsilon.
\end{equation}
Moreover, $\nu(0)=A$ and $\nu'(0) = K$ in $\mathcal{S}^p(\mathcal{H})$, so there exists $\mu_2 > 0$ such that for any $|t| < \mu_2$,
\begin{equation*}
\left\| \frac{(A+tK) - \mu(t)}{t} \right\|_p = \left\| \frac{\mu(t)-A}{t} -K \right\|_p \leq \epsilon.
\end{equation*}
We have, by Proposition $\ref{simplification0}$,
\begin{align*}
& \dfrac{\psi(t) - \tilde{\psi}(t)}{t}  \\
& = \dfrac{1}{t} \sum_{k=1}^n \left[ \Gamma^{(A+tK)^{n-k+1}, (\nu(t))^{k-1}}(f^{[n-1]})-\Gamma^{(A+tK)^{n-k}, (\nu(t))^k}(f^{[n-1]}) \right](K, \ldots, K) \\
& = \sum_{k=1}^{n} \left[ \Gamma^{(A+tK)^{n-k+1}, (\nu(t))^k}(f^{[n]})\right]\left( (K)^{n-k}, \frac{(A+tK) - \mu(t)}{t} , (K)^{k-1} \right).
\end{align*}
Hence, by Remark $\ref{Ineqmulti}$ and Theorem $\ref{BddnessMOI}$, there exists a constant $\alpha > 0$ depending only on $p, n, \|f^{(n)}\|_{\infty}$ and $\|K\|_p$ such that, for any $t\in \mathbb{R}$,
\begin{equation}\label{nu2}
\| \psi(t) - \tilde{\psi}(t) \|_p \leq |t| n \alpha \epsilon.
\end{equation}
Finally, by \eqref{nu1} and \eqref{nu2} and by triangle inequality we have, noting that $\psi(0) = \tilde{\psi}(0)$,
\begin{align*}
\| \psi(t) - \psi(0) - t \tilde{\psi}'(0) \|_p
& \leq \| \psi(t) -  \tilde{\psi}(t)\|_p + \| \tilde{\psi}(t) - \tilde{\psi}(0) - t \tilde{\psi}'(0) \|_p \\
& \leq (n\alpha + 1) |t| \epsilon,
\end{align*}
for any $|t| < \text{min}(\mu_1, \mu_2)$, which proves the claim.

\end{proof}

The following lemma will allow us to reduce the question of differentiability of $\varphi$ defined in \eqref{defofvarphi} for an unbounded operator $A$ to the question of differentiability for a bounded operator.

\begin{lemma}\label{reducbdd}
Let $1 < p < \infty$, $A,K,Y$ be selfadjoint operators in $\mathcal{H}$ with $K$ bounded and $Y\in \mathcal{S}^p(\mathcal{H})$. Let $n \in\N, n\geq 1$ and let $f$ be $n$-times differentiable on $\mathbb{R}$ with $f^{(n)}$ bounded. Let $m\geq 1$ be an integer. We let $E_m = \chi_{[-m,m]}(A), A_m = AE_m, K_m = E_m K E_m$ and $Y_m = E_m Y E_m$. Then
\begin{equation}\label{reducbddeq}
\left[\Gamma^{A+K_m, \ldots, A+K_m}(f^{[n]}) \right](Y_m, \ldots, Y_m) 
= \left[\Gamma^{A_m+K_m, \ldots, A_m+K_m}(f^{[n]}) \right](Y_m, \ldots, Y_m).
\end{equation}
\end{lemma}

\begin{proof}
We first assume that $Y\in \mathcal{S}^2(\mathcal{H})$. Note that the projection $E_m$ commutes with $A + K_m$ so that for any $g\in C_b(\mathbb{R})$ we have, by \cite[(7.25)]{KPSS},
$$
E_m g(A + K_m) = g(A_m + K_m) = g(A + K_m) E_m .
$$
From this equality, we easily deduce that for any $\phi \in C_b(\mathbb{R}) \otimes \cdots \otimes C_b(\mathbb{R})$,
\begin{equation}\label{reducbddeq2}
\left[\Gamma^{A+K_m, \ldots, A+K_m}(\phi) \right](Y_m, \ldots, Y_m) 
= \left[\Gamma^{A_m+K_m, \ldots, A_m+K_m}(\phi) \right](Y_m, \ldots, Y_m).
\end{equation}
By approximation, this implies that \eqref{reducbddeq2} holds true whenever $\phi$ belongs to the uniform closure of $C_b(\mathbb{R})\otimes \cdots \otimes  C_b(\mathbb{R})$, which contains in particular $C_0(\mathbb{R}^{n+1})$.

Assume now that $\phi \in C_b(\mathbb{R}^{n+1})$. Let
$(g_k)_{k\geq 1}$ be a sequence of functions in
$C_0(\mathbb{R})$ satisfying the following two properties:
$$
\forall\, k\in\N, \quad 0\leq g_k\leq 1,
\qquad\hbox{and}\qquad\forall\, r\in \mathbb{R},\quad
g_k(r)\overset{k\to\infty}{\longrightarrow} 1.
$$
For any $k\geq 1, \phi g_k \in C_0(\mathbb{R}^{n+1})$, so the latter implies that
\begin{equation}\label{reducbddeq1}
\left[\Gamma^{A+K_m, \ldots, A+K_m}(\phi g_k) \right](Y_m, \ldots, Y_m) 
= \left[\Gamma^{A_m+K_m, \ldots, A_m+K_m}(\phi g_k) \right](Y_m, \ldots, Y_m).
\end{equation}
By Lebesgue's dominated convergence theorem, the properties satisfied by the sequence $(g_k)_{k\geq 1}$ imply that $(\phi g_k)_{k\geq 1}$ converges to $\phi$ for the $w^*$-topology of $L^{\infty}\left(\prod_{i=1}^n \lambda_{A + K_m}\right)$ and $L^{\infty}\left(\prod_{i=1}^n \lambda_{A_m + K_m}\right)$. Hence, by the $w^*$-continuity of multiple operator integrals, we obtain, by taking the limit on $k$ in \eqref{reducbddeq1},
\begin{equation}\label{reducbddeq3}
\left[\Gamma^{A+K_m, \ldots, A+K_m}(\phi) \right](Y_m, \ldots, Y_m) 
= \left[\Gamma^{A_m+K_m, \ldots, A_m+K_m}(\phi) \right](Y_m, \ldots, Y_m).
\end{equation}

For any $k\geq 1$, let $\phi_k = \varphi_{n,g_k}$ as defined in the proof of Theorem $\ref{BddnessMOI}$. Then $(\varphi_{n,g_k})_{k\geq 1} \subset C_b(\mathbb{R}^{n+1})$ and the sequence $w^*$-converges to $f^{[n]}$ for the $w^*$-topologies of $L^{\infty}\left(\prod_{i=1}^n \lambda_{A + K_m}\right)$ and $L^{\infty}\left(\prod_{i=1}^n \lambda_{A_m + K_m}\right)$. Hence, $\phi_k$ satisfies \eqref{reducbddeq3} for any $k\geq 1$ and by the $w^*$-continuity of multiple operator integrals, we get that $\phi$ satisfies \eqref{reducbddeq}. 

In the case $1 < p < \infty$, we approximate $Y \in \mathcal{S}^p(\mathcal{H})$ by a sequence $(Y_j)_{j\geq 1}$ of elements of $\mathcal{S}^2(\mathcal{H}) \cap \mathcal{S}^p(\mathcal{H})$ and then pass to the limit in the equality
\begin{align*}
&\left[\Gamma^{A+K_m, \ldots, A+K_m}(f^{[n]}) \right]((Y_j)_m, \ldots, (Y_j)_m) \\
& \ \ \ \ = \left[\Gamma^{A_m+K_m, \ldots, A_m+K_m}(f^{[n]}) \right]((Y_j)_m, \ldots, (Y_j)_m)
\end{align*}
as $j \to \infty$, using the estimate in Theorem $\ref{BddnessMOI}$ and the fact that $(Y_j)_m \underset{j \to \infty}{\longrightarrow} Y_m$ in $\mathcal{S}^p(\mathcal{H})$.
\end{proof}

\subsection{Proofs of the main results}\label{ProofofMR}

We now turn to the proof of the main results of this paper, stated in Subsection $\ref{SecMR}$.

\begin{proof}[Proof of Theorem $\ref{FormulaSA}$]
The assumptions on $f$ ensure, by \cite[Theorem 7.13]{KPSS}, that $\varphi$ is differentiable on $\mathbb{R}$ and that for any $t\in \mathbb{R}$,
$$
\varphi'(t) = \left[\Gamma^{A+tK, A+tK}(f^{[1]}) \right](K).
$$
Assume now that $\varphi$ is $(n-1)$-times differentiable on $\mathbb{R}$ with
$$
\dfrac{\varphi^{(n-1)}(t)}{(n-1)!} = \left[\Gamma^{A+tK, \ldots, A+tK}(f^{[n-1]}) \right](K, \ldots, K).
$$
We have to show that the function
$$
\psi : t\in \mathbb{R} \mapsto \left[\Gamma^{A+tK, \ldots, A+tK}(f^{[n-1]}) \right](K, \ldots, K)
$$
is differentiable and that for any $t\in \mathbb{R}$,
$$
\psi'(t) = n \left[\Gamma^{A+tK, \ldots, A+tK}(f^{[n]}) \right](K, \ldots, K).
$$
It is clear that we only have to prove the differentiability in $0$, from which we can deduce the differentiabily on $\mathbb{R}$. In this case, by Lemma $\ref{simplification3}$, it is sufficient to prove the differentiability for $K$ belonging to a dense subset of $(\mathcal{S}^p(\mathcal{H}))_{\text{sa}}$. By \cite[Proposition 6.2]{KPSS}, the subspace $X_0$ defined by
$$
X_0 = \left\lbrace i[A,Y] + Z \ \text{with} \ Y,Z \in (\mathcal{S}^p(\mathcal{H}))_{\text{sa}} \ \text{and} \ Z \ \text{commutes with} \ A \right\rbrace \subset (\mathcal{S}^p(\mathcal{H}))_{\text{sa}}
$$
is dense in $(\mathcal{S}^p(\mathcal{H}))_{\text{sa}}$. Let $K = i[A,Y] + Z \in X_0$ and show that
$$
\psi_0 : t\in \mathbb{R} \mapsto \left[\Gamma^{A+tK, \ldots A+tK}(f^{[n-1]}) \right](K, \ldots, K)
$$
is differentiable in $0$ with
$$
\psi_0'(0) = n \left[\Gamma^{A, \ldots, A}(f^{[n]}) \right](K, \ldots, K).
$$
Let $\nu(t) = e^{-itY}(A+tZ)e^{itY}$. We have $\nu(0)=A$ and $\nu$ is $\mathcal{S}^p$-differentiable in $0$ with $\nu'(0)=K$. Hence, by Lemma $\ref{simplification4}$, to prove the latter, it is equivalent to prove that $\tilde{\psi} : \mathbb{R} \to \mathcal{S}^p(\mathcal{H})$ defined by
$$
\tilde{\psi}(t) = \left[ \Gamma^{\nu(t),\ldots,\nu(t)}(f^{[n-1]}) \right](K, \ldots, K), t\in \mathbb{R},
$$
is differentiable in $0$ with
$$
\tilde{\psi}'(0) = n \left[\Gamma^{A, \ldots, A}(f^{[n]}) \right](K, \ldots, K).
$$
We have, by Proposition $\ref{simplification0}$,
\begin{align*}
& \dfrac{\tilde{\psi}(t) - \tilde{\psi}(0)}{t}  \\
& = \dfrac{1}{t} \sum_{k=1}^n \left[ \Gamma^{(\nu(t))^{n-k+1}, (A)^{k-1}}(f^{[n-1]})-\Gamma^{(\nu(t))^{n-k}, (A)^k}(f^{[n-1]}) \right](K, \ldots, K) \\
& = \sum_{k=1}^{n} \left[ \Gamma^{(\nu(t))^{n-k+1}, (A)^k}(f^{[n]})\right]\left( (K)^{n-k}, \frac{\nu(t) - A}{t} , (K)^{k-1} \right).
\end{align*}
For $1\leq k\leq n$ and any $t\neq 0$, let
$$\displaystyle \Gamma_{\nu,k}(t) = \left[ \Gamma^{(\nu(t))^{n-k+1}, (A)^k}(f^{[n]})\right]\left( (K)^{n-k}, \frac{\nu(t) - A}{t} , (K)^{k-1} \right).$$
Since $\frac{\nu(t) - A}{t}$ goes to $K$ in $\mathcal{S}^p(\mathcal{H})$ as $t$ goes to $0$, by uniform boundedness of $\Gamma^{(\nu(t))^{n-k+1}, (A)^k}(f^{[n]}) \in \mathcal{B}_n(\mathcal{S}^p), t\in \mathbb{R}$, we deduce that if one of those limits exists, so does the second one and we have
$$
\underset{t \to 0}{\lim} \ \Gamma_{\nu,k}(t) = \underset{t \to 0}{\lim} \ \left[ \Gamma^{(\nu(t))^{n-k+1}, (A)^k}(f^{[n]})\right]\left( K,\ldots,K \right).
$$
Note that
\begin{align*}
& \left[ \Gamma^{(\nu(t))^{n-k+1}, (A)^k}(f^{[n]})\right]\left( K,\ldots,K \right) \\
& \ \ \ \ \ \ = e^{-itY}\left[ \Gamma^{(A+tZ)^{n-k+1}, (A)^k}(f^{[n]})\right]\left((e^{itY}Ke^{-itY})^{n-k},e^{itY}K, (K)^{k-1} \right).
\end{align*}
In fact, more generally, for any $t\in \mathbb{R}$, for any $g\in \text{Bor}(\mathbb{R}^{n+1})$ such that $\Gamma^{(\nu(t))^{n-k+1}, (A)^k}(g) \in \mathcal{B}_n(\mathcal{S}^p)$ and any $X \in  \mathcal{S}^p(\mathcal{H})$, we have
\begin{align*}
& \left[ \Gamma^{(\nu(t))^{n-k+1}, (A)^k}(g)\right]\left(X, \ldots, X \right) \\
& \ \ \ \ \ \ = e^{-itY}\left[ \Gamma^{(A+tZ)^{n-k+1}, (A)^k}(g)\right]\left((e^{itY}Xe^{-itY})^{n-k},e^{itY}X, (X)^{k-1} \right).
\end{align*}
Indeed, when $g$ is an element of $\text{Bor}(\mathbb{R}) \otimes \cdots \otimes \text{Bor}(\mathbb{R})$, this equality is a consequence of the fact that for any $h \in \text{Bor}(\mathbb{R})$, $h(e^{-itY}(A+tZ)e^{itY}) = e^{-itY}h(A+tZ)e^{itY}$. Hence, if $p=2$, the general case follows from the $w^*$-continuity of multiple operator integrals. If $1<p<\infty$, we approximate $X\in \mathcal{S}^p(\mathcal{H})$ by elements of $\mathcal{S}^2(\mathcal{H}) \cap \mathcal{S}^p(\mathcal{H})$. Details are left to the reader.

Now, when $t$ goes to $0$, $e^{-itY} \to 1$ in $\mathcal{B}(\mathcal{H})$ so that $e^{itY}Ke^{-itY}, e^{itY}K \to K$ in $\mathcal{S}^p(\mathcal{H})$. Hence, by uniform boundedness of $\Gamma^{(A+tZ)^{n-k+1}, (A)^k}(f^{[n]}) \in \mathcal{B}_n(\mathcal{S}^p), t\in \mathbb{R}$, we have that if one of those limits exists, so does the second one and then
\begin{align*}
& ~ \ \ \ \underset{t \to 0}{\lim} \ e^{-itY}\left[ \Gamma^{(A+tZ)^{n-k+1}, (A)^k}(f^{[n]})\right]\left((e^{itY}Ke^{-itY})^{n-k},e^{itY}K, (K)^{k-1} \right) \\
& = \underset{t \to 0}{\lim} \ \left[ \Gamma^{(A+tZ)^{n-k+1}, (A)^k}(f^{[n]})\right]\left(K, \ldots, K \right).
\end{align*}
Define
$$
\xi(t) = \displaystyle \sum_{k=1}^n \left[ \Gamma^{(A+tZ)^{n-k+1}, (A)^k}(f^{[n]})\right] (K, \ldots, K), t \in \mathbb{R}.
$$
The latter implies that if $\xi$ has a limit in $0$, then so does $\dfrac{\tilde{\psi}(t) - \tilde{\psi}(0)}{t}$ with the same limit. Hence, in order to prove Formula \eqref{diffenrential}, we have to show that $\xi$ has a limit in $0$ and that
$$
\underset{t \to 0}{\lim} \ \xi(t) = n \left[\Gamma^{A, \ldots, A}(f^{[n]}) \right](K, \ldots, K) \ \ \text{in} \ \ \mathcal{S}^p(\mathcal{H}).
$$
For any $1\leq k\leq n$ and any $t\in \mathbb{R}$, let $\Gamma_k(t) = \Gamma^{(A+tZ)^{n-k+1}, (A)^k}(f^{[n]}).$
Since $K = i[A,Y] + Z$, we have, for any $1\leq k \leq n$,
\begin{align*}
\left[\Gamma_k(t) \right](K,\ldots,K)
&  = \left[ \Gamma_k(t) \right] (Z, \ldots, Z) \\
&  + \sum_{j=1}^n \ \sum_{\substack{K_m = i[A,Y] \\\text{or} \ K_m=Z, \\ 1 \leq m \leq n-1}} i \left[ \Gamma_k(t) \right] (K_1, \ldots, K_{j-1}, [A,Y], K_j, \ldots, K_{n-1})
\end{align*}

Hence,
\begin{align*}
\xi(t) 
& = \sum_{k=1}^{n} \left[ \Gamma_k(t) \right] (Z, \ldots, Z) \\
& + \sum_{k=1}^{n} \sum_{j=1}^n \ \sum_{\substack{K_m = i[A,Y] \\\text{or} \ K_m=Z, \\ 1 \leq m \leq n-1}} i \left[ \Gamma_k(t) \right] (K_1, \ldots, K_{j-1}, [A,Y], K_j, \ldots, K_{n-1}).
\end{align*}
By Lemma $\ref{simplification5}$ and Lemma $\ref{MOIcomm}$ we have
\begin{align*}
\sum_{k=1}^{n} \left[ \Gamma_k(t) \right] (Z, \ldots, Z)
& = \dfrac{1}{t} \left[ \Gamma^{A+tZ,\ldots, A+tZ}(f^{[n-1]}) - \Gamma^{A,\ldots, A}(f^{[n-1]}) \right](Z, \ldots, Z) \\
& = \dfrac{1}{(n-1)!} \ \dfrac{f^{(n-1)}(A+tZ)-f^{(n-1)}(A) }{t} \ Z^{n-1}.
\end{align*}
By \cite[Lemma 3.4 (ii)]{KS}, this quantity converges as $t$ goes to $0$ in $\mathcal{S}^p(\mathcal{H})$ to
$$
\dfrac{1}{(n-1)!} \ f^{(n)}(A) Z^n.
$$
Since $n! f^{[n]}(x,\ldots,x) = f^{(n)}(x)$, the latter is in turn, by Lemma $\ref{MOIcomm}$, equal to
$$
n \left[\Gamma^{A,\ldots, A}(f^{[n]}) \right](Z, \ldots, Z).
$$

We will now show that for any $1\leq j,k \leq n$ and for any $K_1, \ldots, K_{n-1}$ with $K_m= Z$ or $i[A,Z]$, $1\leq m \leq n-1$, 
$$
\left[\Gamma_k(t)\right](K_1, \ldots, K_{j-1}, [A,Y], K_j, \ldots, K_{n-1})
$$
goes to $\left[\Gamma^{A,\ldots, A}(f^{[n]}) \right](K_1, \ldots, K_{j-1}, [A,Y], K_j, \ldots, K_{n-1})$ in $\mathcal{S}^p(\mathcal{H})$ as $t$ goes to $0$.

Assume first that $n-k+2 \leq j \leq n$. Since $A$ and $Z$ are bounded operators, we have, by Remark $\ref{approximatebdd}$,
\begin{align*}
& \left[\Gamma_k(t)\right](K_1, \ldots, K_{j-1}, [A,Y], K_j, \ldots, K_{n-1}) \\
& = \left[\Gamma^{(A+tZ)^{n-k+1},(A)^{k-1}}(f^{[n-1]}) \right](K_1, \ldots, K_{j-1}, YK_j, K_{j+1}, \ldots, K_{n-1}) \\
& \ \ \ - \left[\Gamma^{(A+tZ)^{n-k+1},(A)^{k-1}}(f^{[n-1]}) \right](K_1, \ldots, K_{j-1}, YK_{j-1}, K_j, \ldots, K_{n-1}),
\end{align*}
with a simple modification in the case $j=n$. By Lemma $\ref{simplification5}$ the latter converges, as $t$ goes to $0$, to
\begin{align*}
&\left[\Gamma^{A, \ldots, A}(f^{[n-1]}) \right](K_1, \ldots, K_{j-1}, YK_j, K_{j+1}, \ldots, K_{n-1}) \\
& - \left[\Gamma^{A, \ldots, A}(f^{[n-1]}) \right](K_1, \ldots, K_{j-1}, YK_{j-1}, K_j, \ldots, K_{n-1}),
\end{align*}
which is in turn equal to $\left[\Gamma^{A,\ldots, A}(f^{[n]}) \right](K_1, \ldots, K_{j-1}, [A,Y], K_j, \ldots, K_{n-1})$.

Assume now that $j=n-k+1$. In this case, by Remark $\ref{approximatebdd}$,
\begin{align*}
& \left[\Gamma^{(A+tZ)^{n-k+1},(A)^k}(f^{[n]}) \right](K_1, \ldots, K_{j-1}, (A+tZ)Y - YA, K_j, \ldots, K_{n-1}) \\
& = \left[\Gamma^{(A+tZ)^{n-k+1},(A)^{k-1}}(f^{[n-1]}) \right](K_1, \ldots, K_{j-1}, YK_j, K_{j+1}, \ldots, K_{n-1}) \\
& \ \ \ - \left[\Gamma^{(A+tZ)^{n-k+1},(A)^{k-1}}(f^{[n-1]}) \right](K_1, \ldots, K_{j-1}, YK_{j-1}, K_j, \ldots, K_{n-1}).
\end{align*}
Since $(A+tZ)Y - YA = [A,Y] +tZY$, we get
\begin{align*}
& \left[\Gamma_k(t)\right](K_1, \ldots, K_{j-1}, [A,Y], K_j, \ldots, K_{n-1}) \\
& = \left[\Gamma^{(A+tZ)^{n-k+1},(A)^{k-1}}(f^{[n-1]}) \right](K_1, \ldots, K_{j-1}, YK_j, K_{j+1}, \ldots, K_{n-1}) \\
& \ \ \ - \left[\Gamma^{(A+tZ)^{n-k+1},(A)^{k-1}}(f^{[n-1]}) \right](K_1, \ldots, K_{j-1}, YK_{j-1}, K_j, \ldots, K_{n-1}) \\
& \ \ \ - t \left[\Gamma^{(A+tZ)^{n-k+1},(A)^k}(f^{[n]}) \right](K_1, \ldots, K_{j-1}, ZY, K_j, \ldots, K_{n-1}).
\end{align*}
The inequality
\begin{align*}
& \| t \left[\Gamma^{(A+tZ)^{n-k+1},(A)^k}(f^{[n]}) \right](K_1, \ldots, K_{j-1}, ZY, K_j, \ldots, K_{n-1}) \|_p \\
& \leq |t| c_{p,n} \|f^{(n)}\|_{\infty} \|K_1\|_p \ldots \|K_{n-1}\|_p \|ZY\|_p
\end{align*}
and the same reasoning as for the case $n-k+2 \leq j \leq n$ show that
$$
\left[\Gamma_k(t)\right](K_1, \ldots, K_{j-1}, [A,Y], K_j, \ldots, K_{n-1})
$$
converges to $\left[\Gamma^{A,\ldots, A}(f^{[n]}) \right](K_1, \ldots, K_{j-1}, [A,Y], K_j, \ldots, K_{n-1})$.

Finally, assume that $1\leq j \leq n-k$. By Remark $\ref{approximatebdd}$,
\begin{align*}
& \left[\Gamma^{(A+tZ)^{n-k+1},(A)^k}(f^{[n]}) \right](K_1, \ldots, K_{j-1}, [A+tZ,Y], K_j, \ldots, K_{n-1}) \\
& = \left[\Gamma^{(A+tZ)^{n-k+1},(A)^{k-1}}(f^{[n-1]}) \right](K_1, \ldots, K_{j-1}, YK_j, K_{j+1}, \ldots, K_{n-1}) \\
& \ \ \ - \left[\Gamma^{(A+tZ)^{n-k+1},(A)^{k-1}}(f^{[n-1]}) \right](K_1, \ldots, K_{j-1}, YK_{j-1}, K_j, \ldots, K_{n-1}),
\end{align*}
with a simple modification in the case $j=1$ as in Remark $\ref{approximatebdd}$. Note that $[A+tZ,Y] = [A,Y] + t[Z,Y]$ and then reason as in the previous case to show that
$$
\left[\Gamma_k(t)\right](K_1, \ldots, K_{j-1}, [A,Y], K_j, \ldots, K_{n-1})
$$
goes to $\left[\Gamma^{A,\ldots, A}(f^{[n]}) \right](K_1, \ldots, K_{j-1}, [A,Y], K_j, \ldots, K_{n-1})$ in $\mathcal{S}^p(\mathcal{H})$ as $t$ goes to $0$.
Hence, we proved that
\begin{align*}
& \underset{t \to 0}{\lim} \ \xi(t) \\
& = n \left[\Gamma^{A,\ldots, A}(f^{[n]}) \right](Z, \ldots, Z) \\
& \ \ \ \ + \sum_{k=1}^n \sum_{j=1}^n \ \sum_{\substack{K_m = i[A,Y] \\\text{or} \ K_m=Z, \\ 1 \leq m \leq n-1}} i \left[ \Gamma^{A, \ldots, A}(f^{[n]}) \right] (K_1, \ldots, K_{j-1}, [A,Y], K_j, \ldots, K_{n-1}) \\
& = n \left[\Gamma^{A,\ldots, A}(f^{[n]}) \right](K, \ldots, K).
\end{align*}
This proves that $\varphi \in D^n(\mathbb{R},\Sc^p(\H))$ and that for any $1\leq k \leq n$, $\varphi^{(k)}$ is given by \eqref{diffenrential}.

Finally, let $I\subset \mathbb{R}$ be a bounded interval and let $1 \leq i \leq n-1$. We let $J \subset \mathbb{R}$ to be a bounded interval such that, for any $t\in I$, $\sigma(A+tK) \subset J$. There exists $f_i \in C^i(\mathbb{R})$ compactly supported such that $f_i = f$ on $J$. Then for any $t\in I$,
\begin{align*}
\dfrac{1}{k!} \varphi^{(i)}(t)
& = \left[ \Gamma^{A+tK, A+tK, \ldots, A+tK}((f^{[i]})_{|J^{n+1}}) \right] (K, \ldots, K) \\
& = \left[ \Gamma^{A+tK, A+tK, \ldots, A+tK}(f_i^{[i]}) \right] (K, \ldots, K).
\end{align*}

Hence, since $f_i^{(i)}$ is bounded on $\mathbb{R}$, $\varphi^{(i)}$ is bounded on $I$ by Theorem $\ref{BddnessMOI}$. Similarly, since $f^{(n)}$ is bounded on $\mathbb{R}$, $\varphi^{(n)}$ is bounded on $\mathbb{R}$.

\end{proof}

We now turn to the proof of Theorem $\ref{FormulaSAunbdd}$.

\begin{proof}[Proof of Theorem $\ref{FormulaSAunbdd}$]
By \cite[Theorem 7.18]{KPSS}, $\varphi$ is differentiable on $\mathbb{R}$ and for any $t\in \mathbb{R}$, $\varphi'(t) = \left[\Gamma^{A+tK, A+tK}(f^{[1]}) \right](K).$ 
Assume now that $\varphi$ is $(n-1)$-times differentiable on $\mathbb{R}$ with
$$
\dfrac{\varphi^{(n-1)}(t)}{(n-1)!} = \left[\Gamma^{A+tK, \ldots, A+tK}(f^{[n-1]}) \right](K, \ldots, K).
$$
We will prove that the function
$$
\psi : t\in \mathbb{R} \mapsto \left[\Gamma^{A+tK, \ldots, A+tK}(f^{[n-1]}) \right](K, \ldots, K)
$$
is differentiable in $0$ and that
\begin{equation}\label{unbbproof}
\psi'(0) = n \left[\Gamma^{A, \ldots, A}(f^{[n]}) \right](K, \ldots, K).
\end{equation}

For any $m\geq 1$, let $E_m = \chi_{[-m,m]}(A)$. Then $A_m := AE_m$ is bounded. Note that $(E_m)_{m\geq 1}$ converges strongly to the identity so for any $K\in \mathcal{S}^p(\mathcal{H})$, $K_m:=E_m K E_m$ converges to $K$ in $\mathcal{S}^p(\mathcal{H})$ as $m$ goes to $\infty$. This implies that the set
$$
X_0 = \left\lbrace K_m \ | \ K \in (\mathcal{S}^p(\mathcal{H}))_{\text{sa}}, m\geq 1 \right\rbrace \subset (\mathcal{S}^p(\mathcal{H}))_{\text{sa}}
$$
is dense in $(\mathcal{S}^p(\mathcal{H}))_{\text{sa}}$. Hence, by Lemma $\ref{simplification3}$, we only have to prove \eqref{unbbproof} for $K$ element of $X_0$. Let $K=K_m \in X_0$ for some $m\geq 1$.
By Lemma $\ref{reducbdd}$, we have, for any $t\in \mathbb{R}$,
$$
\psi(t) = \left[\Gamma^{A_m+tK_m, \ldots, A_m+tK_m}(f^{[n-1]}) \right](K_m, \ldots, K_m)
$$
which is, by Theorem $\ref{FormulaSA}$, differentiable in $0$ with
$$
\psi'(0) = n\left[\Gamma^{A_m, \ldots, A_m}(f^{[n]}) \right](K_m, \ldots, K_m).
$$
Using Lemma $\ref{reducbdd}$ again, we see that
$$
\psi'(0) = n\left[\Gamma^{A, \ldots, A}(f^{[n]}) \right](K_m, \ldots, K_m).
$$
This proves that $\varphi \in D^n(\mathbb{R},\Sc^p(\H))$ and that the derivatives of $\varphi$ are given by \eqref{diffenrentialunbdd}.

Finally, the boundedness of the derivatives follows from Theorem $\ref{BddnessMOI}$.

\end{proof}

\begin{proof}[Proof of Proposition $\ref{Taylorrmd}$]
The existence of the derivatives $\displaystyle \frac{d^k}{dt^k}\Big(f(A+tK)\Big)\Big|_{t=0}, k=1, \ldots, n-1,$ are ensured by Theorem $\ref{FormulaSA}$ and Theorem $\ref{FormulaSAunbdd}$. The representation \eqref{Taylor} can be obtained by induction on $n$, using Theorem $\ref{BddnessMOI}$. See the proof of \cite[Theorem 4.1 (ii)]{CLSS} for more details. From this representation, we obtain the estimate \eqref{Taylorineq} by applying Inequality \eqref{IneqDDn}.

\end{proof}

\begin{proof}[Proof of Proposition $\ref{Contcase}$]
By Theorem $\ref{FormulaSA}$, we know that $\varphi$ is $n$-times differentiable and that for any $t\in \mathbb{R}$,
\begin{equation*}
\dfrac{1}{n!} \varphi^{(n)}(t) = \left[ \Gamma^{A+tK, A+tK, \ldots, A+tK}(f^{[n]}) \right] (K, \ldots, K).
\end{equation*}
Hence, to prove the result, we have to show that the mapping
$$
t \in \mathbb{R} \mapsto \left[ \Gamma^{A+tK, A+tK, \ldots, A+tK}(f^{[n]}) \right] (K, \ldots, K) \in \mathcal{S}^p(\mathcal{H})
$$
is continuous. More generally, we will prove that for any $X = (X_1, \ldots, X_n) \in \mathcal{S}^p(\mathcal{H})^n$, the mapping
$$
\varphi_X : t \in \mathbb{R} \mapsto \left[ \Gamma^{A+tK, A+tK, \ldots, A+tK}(f^{[n]}) \right] (X_1, \ldots, X_n) \in \mathcal{S}^p(\mathcal{H})
$$
is continuous. Note that it is sufficient to prove that $\varphi_X$ is continuous in $0$.

Let $h \in C^n(\mathbb{R})$.  Assume that either $A$ is bounded or $h^{(i)}$ is bounded for all $1\leq i\leq n-1$ and further assume that $h^{(n)} \in C_0(\mathbb{R})$. It follows from \cite[Theorem 3.4]{LMS} that $h$ is $n$-times continuously Fr\'echet $\mathcal{S}^p$-differentiable at $A$. This implies that for any $X_1, \ldots, X_n \in \mathcal{S}^p(\mathcal{H})$, the mapping
\begin{equation}\label{casefnC0}
t\in \mathbb{R} \mapsto \left[ \Gamma^{A+tK, A+tK, \ldots, A+tK}(h^{[n]}) \right] (X_1, \ldots, X_n)
\end{equation}
is continuous in $0$. Hence, if $f^{(n)} \in C_0(\mathbb{R})$, $\varphi \in C^n(\mathbb{R},\Sc^p(\H))$. The rest of the proof consists in reducing to this particular case.\\

Let $(g_k)_{k\geq 1}$ be a sequence of $C^{\infty}_c(\mathbb{R})$ satisfying the following two properties:
$$
\forall\, k\in\N, \quad 0\leq g_k\leq 1,
\qquad\hbox{and}\qquad\forall\, r\in \mathbb{R},\quad
g_k(r)\overset{k\to\infty}{\longrightarrow} 1.
$$
Let $k\geq 1$. Define $G_k^n = g_k \otimes \underbrace{1 \otimes \cdots \otimes 1}_{n-1 \ \text{times}} \otimes g_k$ and write
$$
G_k^n(x_1, \ldots, x_n) = h_1(x_1,x_2) h_2(x_{n-1},x_n)
$$
where $h_1 = g_k \otimes 1$ and $h_2 = 1 \otimes g_k$. We have
$$
\left[\Gamma^{A+tK, A+tK}(h_1)\right](X_1) = g_k(A+tK) X_1
$$
and
$$\left[\Gamma^{A+tK, A+tK}(h_2)\right](X_n) = X_n g_k(A+tK).
$$
Hence, by Lemma $\ref{simplification1}$, the function
$$
\varphi_{k,X}^n : t \in \mathbb{R} \mapsto \left[ \Gamma^{A+tK, A+tK, \ldots, A+tK}(f^{[n]}G_k^n) \right] (X_1, \ldots, X_n) \in \mathcal{S}^p(\mathcal{H})
$$
satisfies, for any $t\in \mathbb{R}$,
$$
\varphi_{k,X}^n(t) = \left[ \Gamma^{A+tK, A+tK, \ldots, A+tK}(f^{[n]}) \right](g_k(A+tK) X_1, X_2, \ldots, X_{n-1}, X_n g_k(A+tK)).
$$
$(g_k)_{k\geq 1}$ is bounded and pointwise convergent to $1$ so $(g_k(A))_{k\geq 1}$ converges strongly to the identity of $\mathcal{H}$ and hence $g_k(A) X_1 \to X_1$ and $X_n g_k(A) \to X_n$ in $\mathcal{S}^p(\mathcal{H})$ as $k\to \infty$. Moreover, it follows from the arguments of the proof of \cite[Lemma 3.4]{CLSS} that $A + tK \to A$ resolvent strongly as $t \to 0$. This means that for any $u \in C_b(\mathbb{R})$,
$$
u(A+tK) \underset{t \to 0}{\longrightarrow} u(A) \ \ \text{strongly}.
$$
For any $k\geq 1, g_k \in C_b(\mathbb{R})$ so $g_k(A+tK) \to g_k(A)$ strongly as $t \to 0$, which implies that $g_k(A+tK) X_1 \to g_k(A) X_1$ and $X_n g_k(A+tK) \to X_n g_k(A)$ in $\mathcal{S}^p(\mathcal{H})$ as $t\to 0$.

Let $\epsilon > 0$. By the latter, there exists $k_0 \in \mathbb{N}$ such that, for any $k\geq k_0$,
$$
\| X_1 - g_{k_0}(A) K\|_p \leq \epsilon \ \ \text{and} \ \ \| X_n - X_n g_{k_0}(A) \|_p \leq \epsilon
$$
and there exists $t_0 > 0$ such that for any $|t| \leq t_0$,
$$
\| g_{k_0}(A)X_1 - g_{k_0}(A+tK) X_1\|_p \leq \epsilon \ \ \text{and} \ \ \| X_n g_{k_0}(A) - X_n g_{k_0}(A+tK) \|_p \leq \epsilon.
$$
Hence, for any $|t| \leq t_0$,
\begin{align*}
& \| X_1 - g_{k_0}(A+tK) X_1\|_p \\
& \leq \| X_1 - g_{k_0}(A) X_1\|_p + \| g_{k_0}(A)X_1 - g_{k_0}(A+tK) X_1\|_p \\
& \leq 2 \epsilon,
\end{align*}
and similarly,
$$
\| X_n - X_n g_{k_0}(A+tK) \|_p \leq 2\epsilon.
$$
By Remark $\ref{Ineqmulti}$ and Theorem $\ref{BddnessMOI}$, there exists a constant $C > 0$ depending only on $p, n, \|f^{(n)}\|_{\infty}$ and $\|K\|_p$ such that, for any $|t| < t_0$,
$$
\| \varphi_X(t) - \varphi_{k_0,X}^n(t) \|_p \leq C \epsilon.
$$
By the triangle inequality we get that for any $|t| < t_0$,
\begin{align*}
& \| \varphi_X(t) - \varphi_X(0) \|_p \\
& \leq \| \varphi_X(t) - \varphi_{k_0,X}^n(t) \|_p + \|\varphi_{k_0,X}^n(t) - \varphi_{k_0,X}^n(0) \|_p + \| \varphi_{k_0,X}^n(0) - \varphi_X(0)\|_p \\
& \leq 2C \epsilon + \|\varphi_{k_0,X}^n(t) - \varphi_{k_0,X}^n(0) \|_p.
\end{align*}
Hence, to prove the result, it suffices to prove that for any $k\geq 1$ and any $X \in \mathcal{S}^p(\mathcal{H})^n, \varphi_{k,X}^n$ is continuous in $0$.\\

Fix $k\geq 1$ and let $g = g_k$. We will prove the continuity of $\varphi_{k,X}^n$ in $0$ by induction on $n$. For $n=1$, we have, for any $(x_0,x_1)\in \mathbb{R}^2$ with $x_0 \neq x_1$,
\begin{align}\label{calculsn1}
\begin{split}
f^{[1]}(x_0,x_1)g(x_0)g(x_1)
& = \dfrac{g(x_1)(gf)(x_0) - (gf)(x_1) g(x_0)}{x_0 - x_1} \\
& = \dfrac{g(x_1)(gf)(x_0) - g(x_1)(gf)(x_1)}{x_0 - x_1} \\
& \ \ \ \ + \dfrac{(gf)(x_1)g(x_1) - (gf)(x_1) g(x_0)}{x_0 - x_1} \\
& = g(x_1) (gf)^{[1]}(x_0,x_1) - (gf)(x_1)g^{[1]}(x_0,x_1).
\end{split}
\end{align}
By continuity, this equality holds true for any $x_0,x_1 \in \mathbb{R}$. Hence, by Lemma $\ref{simplification1}$, we have, for any $t\in \mathbb{R}$ and any $X\in \mathcal{S}^p(\mathcal{H})$,
\begin{align*}
& \varphi_{k,X}^1(t) \\
& = \left[ \Gamma^{A+tK, A+tK}((gf)^{[1]}) \right](Xg(A+tK)) - \left[ \Gamma^{A+tK, A+tK}(g^{[1]}) \right](X(gf)(A+tK)).
\end{align*}

As explained in the first part of the proof, the mappings $t\in \mathbb{R} \mapsto Xg(A+tK) \in \mathcal{S}^p(\mathcal{H})$ and $t\in \mathbb{R} \mapsto X(gf)(A+tK) \in \mathcal{S}^p(\mathcal{H})$ are continuous in $0$. Note that $g', (gf)' \in C_0(\mathbb{R})$ so that, by continuity of the map defined in \eqref{casefnC0} and the uniform boundedness of the mappings $\Gamma^{A+tK, A+tK}((gf)^{[1]}), \Gamma^{A+tK, A+tK}(g^{[1]}), t\in \mathbb{R}$, we get that
\begin{align*}
\underset{t\to 0}{\lim} ~ \varphi_{k,X}^1(t) = \left[ \Gamma^{A, A}((gf)^{[1]}) \right](Xg(A)) - \left[ \Gamma^{A, A}(g^{[1]}) \right](X(gf)(A)) = \varphi_{k,X}^1(0),
\end{align*}
so that $\varphi_{k,X}^1$ is continuous in $0$.

Now, let $n\geq 2$ and assume that for any $1\leq i \leq n-1$ and any $X=(X_1,\ldots, X_{i}) \in \mathcal{S}^p(\mathcal{H})^i$, $\varphi_{k,X}^{i}$ is continuous in $0$. First, we show by induction the following formula: for every $(x_0,\ldots,x_n) \in \mathbb{R}^{n+1}$,
\begin{align}\label{recfng}
\begin{split}
& f^{[n]}(x_0, \ldots, x_n)g(x_0) \\
& = (gf)^{[n]}(x_0, \ldots, x_n) - f(x_n) g^{[n]}(x_0,\ldots, x_n) - \sum_{l=1}^{n-1} g^{[l]}(x_0, \ldots, x_l)f^{[n-l]}(x_{l} ,\ldots, x_n).
\end{split}
\end{align}
For $n=2$, first note that the computations made in \eqref{calculsn1} give
$$
f^{[1]}(x_0,x_1)g(x_0) = (gf)^{[1]}(x_0,x_1) - f(x_1)g^{[1]}(x_0,x_1)
$$
so that
\begin{align*}
& f^{[2]}(x_0,x_1,x_2)g(x_0) \\
& = \dfrac{f^{[1]}(x_0,x_2)g(x_0) - f^{[1]}(x_1,x_2)g(x_1)}{x_0-x_1} + \dfrac{f^{[1]}(x_1,x_2)g(x_1) - f^{[1]}(x_1,x_2)g(x_0)}{x_0-x_1} \\
& = \dfrac{(gf)^{[1]}(x_0,x_2) - f(x_2)g^{[1]}(x_0,x_2) - (gf)^{[1]}(x_1,x_2) + f(x_2)g^{[1]}(x_1,x_2)}{x_0-x_1} \\
& \ \ \ \ - g^{[1]}(x_0,x_1)f^{[1]}(x_1,x_2) \\
& = (gf)^{[2]}(x_0,x_1,x_2) - f(x_2) g^{[2]}(x_0,x_1,x_2) - g^{[1]}(x_0,x_1)f^{[1]}(x_1,x_2),
\end{align*}
which shows \eqref{recfng} for $n=2$. Assume now that we have \eqref{recfng} at the order $n$ and show that it still holds true at the order $n+1$. We have
\begin{align}\label{recu1}
\begin{split}
& f^{[n+1]}(x_0,\ldots, x_{n+1})g(x_0) \\
& = \dfrac{f^{[n]}(x_0, x_2, \ldots, x_{n+1})g(x_0) - f^{[n]}(x_1,\ldots, x_{n+1})g(x_1)}{x_0-x_1} \\
& \ \ \ \ + \dfrac{f^{[n]}(x_1,\ldots, x_{n+1})g(x_1) - f^{[n]}(x_1,\ldots, x_{n+1})g(x_0)}{x_0-x_1} \\
& = \dfrac{f^{[n]}(x_0, x_2, \ldots, x_{n+1})g(x_0) - f^{[n]}(x_1,\ldots, x_{n+1})g(x_1)}{x_0-x_1} \\
& \ \ \ \ - g^{[1]}(x_0,x_1) f^{[n]}(x_1,\ldots, x_{n+1}).
\end{split}
\end{align}
By assumption, we have
\begin{align}\label{recu2}
\begin{split}
& \dfrac{f^{[n]}(x_0, x_2, \ldots, x_{n+1})g(x_0) - f^{[n]}(x_1,\ldots, x_{n+1})g(x_1)}{x_0-x_1} \\
& = \dfrac{(gf)^{[n]}(x_0, x_2, \ldots, x_{n+1}) - (gf)^{[n]}(x_1, \ldots, x_{n+1})}{x_0-x_1} \\
& \ \ \ \ - f(x_{n+1}) \dfrac{g^{[n]}(x_0,x_2,\ldots, x_{n+1}) - g^{[n]}(x_1,\ldots, x_{n+1})}{x_0-x_1} \\
& \ \ \ \ - \sum_{l=1}^{n-1} \dfrac{g^{[l]}(x_0,x_2, \ldots, x_{l+1})-g^{[l]}(x_1, \ldots, x_{l+1})}{x_0-x_1}f^{[n-l]}(x_{l+1} ,\ldots, x_{n+1}) \\
& = (gf)^{[n+1]}(x_0, \ldots, x_{n+1}) - f(x_{n+1}) g^{[n+1]}(x_0,\ldots, x_{n+1}) \\
& \ \ \ \ - \sum_{l=1}^{n-1} g^{[l+1]}(x_0, \ldots, x_{l+1}) f^{[n-l]}(x_{l+1} ,\ldots, x_{n+1}).
\end{split}
\end{align}
We have
\begin{align*}
& \sum_{l=1}^{n-1} g^{[l+1]}(x_0, \ldots, x_{l+1}) f^{[n-l]}(x_{l+1} ,\ldots, x_{n+1}) + g^{[1]}(x_0,x_1) f^{[n]}(x_1,\ldots, x_{n+1}) \\
& = \sum_{j=2}^{n} g^{[j]}(x_0, \ldots, x_j) f^{[n+1-j]}(x_j ,\ldots, x_{n+1}) + g^{[1]}(x_0,x_1) f^{[n]}(x_1,\ldots, x_{n+1}) \\
& = \sum_{j=1}^{n} g^{[j]}(x_0, \ldots, x_j) f^{[n+1-j]}(x_j ,\ldots, x_{n+1}).
\end{align*}
Hence, by \eqref{recu1} and \eqref{recu2}, Formula \eqref{recfng} is proved at the order $n+1$. Note that the previous computations make sense when $x_0 \neq x_1$ and by continuity, the formula also holds true for $x_0 = x_1$. Let $(x_0, \ldots, x_n)\in \mathbb{R}^{n+1}$. We multiply Formula \eqref{recfng} by $g(x_n)$ and we get
\begin{align*}
f^{[n]}G_k^n(x_0, \ldots, x_n) 
& = g(x_n)(gf)^{[n]}(x_0, \ldots, x_n) - (gf)(x_n) g^{[n]}(x_0,\ldots, x_n) \\
& \ \ \ - g(x_n)\sum_{l=1}^{n-1} g^{[l]}(x_0, \ldots, x_l)f^{[n-l]}(x_{l} ,\ldots, x_n).
\end{align*}
Let $X_1, \ldots, X_n \in \mathcal{S}^p(\mathcal{H})$. Applying the operator
$$
\left[\Gamma^{A+tK, \ldots, A+tK}(\cdot)\right](X_1,\ldots, X_n)
$$
to the previous equality gives, by Lemma $\ref{simplification1}$ and Lemma $\ref{simplificationsep}$,
\begin{align*}
\varphi_{k,X}^n(t) = \varphi_1(t) - \varphi_2(t) - \sum_{l=1}^{n-1} \varphi_{3,l}(t)
\end{align*}
where
$$
\varphi_1(t) = \left[\Gamma^{A+tK, \ldots, A+tK}((gf)^{[n]})\right](X_1,\ldots, X_{n-1}, X_ng(A+tK)),
$$
$$
\varphi_2(t) = \left[\Gamma^{A+tK, \ldots, A+tK}(g^{[n]})\right](X_1,\ldots, X_{n-1}, X_n (gf)(A+tK)),
$$
and for any $1\leq l \leq n-1$,
$$
\varphi_{3,l}(t) = \varphi_{3,l}^1(t)\varphi_{3,l}^2(t)
$$
with 
$$
\varphi_{3,l}^1(t) = \left[\Gamma^{A+tK, \ldots, A+tK}(g^{[l]})\right](X_1,\ldots,X_{l})
$$
and
$$\varphi_{3,l}^2(t) = \left[\Gamma^{A+tK, \ldots, A+tK}(f^{[n-l]})\right](X_{l+1},\ldots, X_{n-1}, X_n g(A+tK)).
$$
The functions $g$ and $gf$ belong to $C_b(\mathbb{R})$ so the mappings $t \in \mathbb{R} \mapsto X_ng(A+tK) \in \mathcal{S}^p(\mathcal{H})$ and $t \in \mathbb{R} \mapsto X_n (gf)(A+tK) \in \mathcal{S}^p(\mathcal{H})$ are continuous in $0$. We have $(gf)^{(n)}, g^{(n)} \in C_0(\mathbb{R})$ so by the continuity of the map defined in \eqref{casefnC0}, we get that $\varphi_1$ and $\varphi_2$ are continuous in $0$.

Now let $1 \leq l \leq n-1$. Since $g^{(l)} \in C_0(\mathbb{R})$, $\varphi_{3,l}^1$ is continuous in $0$. We have $1 \leq n-l \leq n-1$, and by assumption, $\varphi_{k,Y}^{n-l}$ is continuous in $0$ for any $Y \in \mathcal{S}^p(\mathcal{H})^l$. Hence, by composition with the continuous map $t \in \mathbb{R} \mapsto X_ng(A+tK) \in \mathcal{S}^p(\mathcal{H})$, we get that $\varphi_{3,l}^2$ is continuous in $0$, so that $\varphi_{3,l}$ also is. We hence proved that $\varphi_{k,X}$ is continuous in $0$, which concludes the proof of the proposition.

\end{proof}

\vskip 1cm
\noindent
{\bf Acknowledgements.} The author is supported by NSFC (11801573).

\vskip 1cm

\end{document}